\documentclass[final,onefignum,onetabnum]{siamart171218}

\usepackage{braket,amsfonts}

\usepackage{mathrsfs}
 \usepackage{mathtools} 

\definecolor{darkblue}{rgb}{0,0,0.7}

\definecolor{darkgreen}{rgb}{0.01,0.75,0.24}

\def \Ee[#1]{\mathcal{E}^{\text{{#1}}}}
\def\R{\mathbf{R}}

\def\pa[#1,#2]{\frac{\partial {#1}}{\partial {#2}} }

\def\idom[#1,#2,#3]{\int_{#1}\hspace{1pt} {#2} \hspace{1pt} \text{d}{#3}}
\def\res[#1,#2]{\left.{#1}\right|_{#2}}

\def\gt{\rightarrow}
\def\lgt{\downarrow}

\def\var[#1,#2]{\langle \delta \mathcal{E}^{\text{{#1}}}({#2}),v\rangle}
\def\vars[#1,#2,#3]{\langle \delta^2\mathcal{E}^{\text{{#1}}}({#2})v,{#3}\rangle}
\def\vard[#1,#2,#3,#4]{\langle \delta\mathcal{E}^{\text{{#1}}}({#2})-\delta\mathcal{E}^{\text{{#3}}}({#4}),v\rangle}

\def\F{\mathcal{F}}

\def\P{\mathscr{P}}
\def\sX{\mathscr{X}}
\def\sY{\mathscr{Y}}

\def\E{\mathbb{E}}

\def\N{\mathbb{N}}

\newcommand{\balpha}{\boldsymbol{\alpha}}
\newcommand{\bbeta}{\boldsymbol{\beta}}

\newcommand{\W}{\mathcal{W}}

\newcommand{\HH}{\mathcal{H}}

\newcommand{\eps}{\varepsilon}
\newcommand{\wgt}{\rightharpoonup}

\newcommand{\bx}{\mathbf{x}}

\newcommand{\KL}{\mathrm{KL}}

\newcommand{\norm}[1]{\lVert #1 \rVert}
\newcommand{\commentout}[1]{}

\newcommand{\be}{\begin{equation}}
\newcommand{\en}{\end{equation}}
\newcommand{\ben}{\begin{equation*}}
\newcommand{\enn}{\end{equation*}}
\newcommand{\bea}{\begin{aligned}}
\newcommand{\ena}{\end{aligned}}

\def\ba#1\ena{\begin{align}#1\end{align}}
\def\ban#1\enan{\begin{align*}#1\end{align*}}

\newtheorem{assumption}[theorem]{Assumption}
\usepackage{array}

\usepackage[caption=false]{subfig}
\usepackage{pgfplots}

\newsiamthm{claim}{Claim}
\newsiamremark{remark}{Remark}
\newsiamremark{hypothesis}{Hypothesis}
\crefname{hypothesis}{Hypothesis}{Hypotheses}

\usepackage{algorithmic}

\usepackage{graphicx,epstopdf}

\Crefname{ALC@unique}{Line}{Lines}

\usepackage{amsopn}

\usepackage{xspace}
\usepackage{bold-extra}
\usepackage[most]{tcolorbox}

\colorlet{texcscolor}{blue!50!black}
\colorlet{texemcolor}{red!70!black}
\colorlet{texpreamble}{red!70!black}
\colorlet{codebackground}{black!25!white!25}

\lstdefinestyle{siamlatex}{%
  style=tcblatex,
  texcsstyle=*\color{texcscolor},
  texcsstyle=[2]\color{texemcolor},
  keywordstyle=[2]\color{texemcolor},
  moretexcs={cref,Cref,maketitle,mathcal,text,headers,email,url},
}

\tcbset{%
  colframe=black!75!white!75,
  coltitle=white,
  colback=codebackground, % bottom/left side
  colbacklower=white, % top/right side
  fonttitle=\bfseries,
  arc=0pt,outer arc=0pt,
  top=1pt,bottom=1pt,left=1mm,right=1mm,middle=1mm,boxsep=1mm,
  leftrule=0.3mm,rightrule=0.3mm,toprule=0.3mm,bottomrule=0.3mm,
  listing options={style=siamlatex}
}

\newtcblisting[use counter=example]{example}[2][]{%
  title={Example~\thetcbcounter: #2},#1}

\newtcbinputlisting[use counter=example]{\examplefile}[3][]{%
  title={Example~\thetcbcounter: #2},listing file={#3},#1}

\DeclareTotalTCBox{\code}{ v O{} }
{ %fontupper=\ttfamily\color{texemcolor},
  fontupper=\ttfamily\color{black},
  nobeforeafter,
  tcbox raise base,
  colback=codebackground,colframe=white,
  top=0pt,bottom=0pt,left=0mm,right=0mm,
  leftrule=0pt,rightrule=0pt,toprule=0mm,bottomrule=0mm,
  boxsep=0.5mm,
  #2}{#1}

\patchcmd\newpage{\vfil}{}{}{}
\begin{tcbverbatimwrite}{tmp_\jobname_header.tex}
\title{Scaling limit of the Stein variational gradient descent: the mean field regime
\thanks{Submitted to the editors DATE.
\funding{This work was supported in part by the National Science Foundation through grants DMS-1454939 (JL) and DMS-1351653 (JN).}}}

\author{Jianfeng Lu\footnotemark[2]
\and Yulong Lu\footnotemark[2]
\and James Nolen\thanks{Department of Mathematics, Duke University, Durham NC 27708, USA (\email{jianfeng@math.duke.edu}, \email{yulonglu@math.duke.edu},  \email{nolen@math.duke.edu}).}}

\headers{Stein Variational Gradient Descent}{Jianfeng Lu, Yulong Lu, and James Nolen}
\end{tcbverbatimwrite}
\input{tmp_\jobname_header.tex}

\ifpdf
\hypersetup{ pdftitle={Stein Variational Gradient Descent} }
\fi

\begin{document}
\maketitle

%% ------------------------------------------------------------------
%% ABSTRACT
%% ------------------------------------------------------------------
\begin{tcbverbatimwrite}{tmp_\jobname_abstract.tex}
\begin{abstract}
We study an interacting particle system 
in $\mathbf{R}^d$ 
motivated by Stein variational gradient descent
[Q. Liu and D. Wang, NIPS 2016], a deterministic algorithm for 
approximating 
a given probability density with unknown normalization based on particles.
We prove that in the large particle limit
the empirical measure of the particle system converges to a solution of a non-local and nonlinear PDE. 
We also prove global existence, uniqueness and regularity of the solution to the limiting PDE. 
Finally, we prove that the solution to the PDE converges
to the unique invariant solution in long time limit. 
\end{abstract}


\end{tcbverbatimwrite}
\input{tmp_\jobname_abstract.tex}
%% ------------------------------------------------------------------
%% END HEADER
%% ------------------------------------------------------------------

\section{Introduction}

In this paper we study the following interacting particle system in $\R^d$:
\be\label{eq:ps}
\bea
\dot{x}_i(t)& = -\frac{1}{N} \sum_{j=1}^N \nabla K(x_i(t) - x_j(t)) - \frac{1}{N} \sum_{j=1}^N K(x_i(t) - x_j(t))\nabla V(x_j(t)), \\
x_i(0) & = x_i^{0} \in \R^d, \quad \quad i=1,\cdots, N.
\ena
\en
We refer to each of the $N$ functions $x_i(\cdot) \in \R^d$ as a particle. The function $K:\R^d \mapsto \R$ is a smooth, symmetric, and positive definite kernel. The function $V:\R^d \to \R$  is a smooth potential such that $e^{-V(x)}$ is integrable. More specific assumptions about $K$ and $V$ are given below.  

We are interested in the macroscopic behavior of the particle
system \eqref{eq:ps} as $N \gt \infty$ in the framework of  mean field limit. Formally
this mean field limit   
 is described by the following non-local, nonlinear partial differential equation (PDE):
\be\label{eq:mfl}
\bea
& \partial_t \rho  = \nabla \cdot(\rho(K \ast (\nabla  \rho + \nabla V \rho))),\\
&\rho(0, \cdot) = \rho_0(\cdot).
\ena\en
We aim to make a rigorous connection between \eqref{eq:ps} and \eqref{eq:mfl}. 
Specifically, we prove global existence and 
uniqueness of a solution to this initial value problem, 
for $\rho_0$ in the appropriate regularity class, and 
we show that the empirical measure
\[
\mu^N_t = \frac{1}{N} \sum_{i=1}^N \delta_{x_i(t)}
\]
converges as $N \to \infty$ to the solution of \eqref{eq:mfl},
assuming $\mu^N_0$ converges to $\rho_0(x)dx$ in the appropriate sense. 
We also want to study the long-time behavior of solutions 
to the mean field PDE \eqref{eq:mfl}.
It is easy to see that the probability density 
$$
\rho_\infty(x) = e^{-V(x)}/Z
$$
with  $Z = \int e^{-V(x)}dx$ is an invariant solution to \eqref{eq:mfl}. 
Under certain assumptions, we prove that
$\rho(t,\cdot)$ converges weakly to $ \rho_\infty$ as $t \to +\infty$.

\subsection{Motivation} Our interest in the particle system \eqref{eq:ps} is mainly motivated by the recent works by Liu and Wang \cite{LiuWang16, Liu17}, where a time-discretized form of \eqref{eq:ps} was introduced as an algorithm called Stein Variational Gradient Descent (SVGD).  The idea of the algorithm is to transport a set of $N$ particles in $\R^d$ so that their empirical measure $\mu^N$ approximates the target probability measure $\rho_\infty(x)dx = Z^{-1}e^{-V(x)}dx$, with an unknown  normalization factor $Z$.  At discrete times, the particles are updated via the map 
\be\label{eq:map}
x \mapsto T(x)  = x + \eps \varphi(x)
\en
where  $\eps$ is a small time step size and $\varphi$ is a velocity field, which is chosen appropriately  so to have a ``fastest decay'' of the Kullback-Leibler (KL)
divergence between the push-forward measure $ T_\# \mu^N$ and the target $\rho_\infty$. 
Recall that the KL-divergence (or relative entropy) $\KL(\mu || \nu)$ between probability measures $\mu$ and $\nu$ is
$$
\KL(\mu || \nu) =  \int \log\Big(\frac{d\mu}{d\nu}\Big)\frac{d\mu}{d\nu}d\nu
$$
if $\mu$ is absolutely continuous with respect to $\nu$, and we set $\KL(\mu || \nu) = +\infty$ if $\mu$ is singular to $\nu$.
This idea of SVGD can be  formalized as choosing the velocity field $\varphi$ to solve the variational problem
\be\label{eq:varp0}
\sup_{\varphi \in \HH} \Big\{-\partial_\eps \KL(T_\# \mu^N \,||\,\rho_\infty)  |_{\eps = 0} \;\;| \;\;\|\varphi\|_\HH  \leq 1\Big\}
\en
at each time step, where $\HH$ is a suitable space of vector fields. 
It is not clear that \eqref{eq:varp0} is well-defined,
because the measure $T_\# \mu^N$ may be singular with respect to
$\rho_\infty$ and $\KL(T_\# \mu^N \,||\, \rho_\infty) = +\infty$.  
However, as shown in \cite{LiuWang16}, \eqref{eq:varp0}
can be given meaning through the observation that if $\mu$ is absolutely continuous with respect to $\rho$ and $\KL(T_\# \mu \,||\, \rho) < \infty$, then 
\[
-\partial_\eps \KL(T_\# \mu \,||\, \rho) |_{\eps = 0} = \E_{\mu} [S_\rho \varphi],
\]
where $S_\rho$ is the so-called Stein operator defined by
\[
S_\rho\varphi := \nabla \log \rho(x) \cdot \varphi(x) + \nabla \cdot \varphi(x).
\]
In view of \eqref{eq:varp0}, this leads to the definition of Stein discrepancy 
\be\label{eq:varp1}
\mathrm{SD}(\mu, \rho, \HH) := \sup_{\varphi\in \HH}\Big\{\E_{\mu} [S_{\rho}\varphi] \ \Big|\  \|\varphi\|_\HH \leq 1\Big\},
\en
which has the property that
$\mathrm{SD}(\mu, \rho, \HH) \geq 0$ is equal to zero 
if and only if $\mu = \rho$ provided that the space $\HH$ is sufficiently rich.
For the empirical measure
$\mu^N$, the objective function $\E_{\mu^N} [S_{\rho_\infty} \varphi]$ in \eqref{eq:varp1} may be well-defined and finite even though $\KL(T_\# \mu^N \,||\, \rho_\infty) = +\infty$.  Furthermore, \cite{LiuWang16} showed that if the space $\HH$ is chosen to be a reproducing kernel Hilbert space with a positive definite kernel $K$, then the velocity field optimizing \eqref{eq:varp1} can be
characterized explicitly and is given by 
\[
\varphi^\ast_{\mu, \rho} (\cdot) \propto \E_{x\sim \mu}[ S_\rho K(x,\cdot)] = \int_{\R^d} \left( \nabla \log \rho(x) K(x,\cdot)  + \nabla_x K(x,\cdot) \right)\mu(dx).
\]
Therefore, interpreting \eqref{eq:varp0} by \eqref{eq:varp1} and using 
the fact that $\rho_\infty(x) \propto  e^{-V(x)}$, one sees that
the optimal solution of \eqref{eq:varp0} is given by 
\be\label{eq:varp2}\bea
\varphi^\ast_{\mu^N, \rho_\infty}(x) & = \E_{x\sim \mu^N}[ S_{\rho_\infty} K(x,\cdot)]\\
& =  -\frac{1}{N} \sum_{j=1}^N \nabla K(x - x_j) - \frac{1}{N} \sum_{j=1}^N K(x - x_j)\nabla V(x_j).
\ena
\en
Putting this optimal velocity back into \eqref{eq:map} and letting the step size
$\eps \lgt 0$ gives the evolution \eqref{eq:ps}.

The variational picture described above about the particle system 
\eqref{eq:ps} suggests that the mean field limit \eqref{eq:mfl} might also
admit a variational structure. Indeed, 
it has been shown heuristically in \cite{Liu17} that 
equation \eqref{eq:mfl} can be viewed 
formally as a gradient flow for the KL-divergence functional  
\[
\rho \mapsto  \KL(\rho \,||\, \rho_\infty) = \int_{\R^d} \rho \log \frac{\rho}{\rho_\infty} \,dx,
\] 
with respect to a generalized optimal transport metric
whose definition involves the 
reproducing kernel Hilbert space with kernel $K(x)$. 
This in particular implies that the KL-divergence functional is a  Lyapunov functional for the PDE \eqref{eq:mfl}, namely 
$$
\frac{d}{dt} \KL(\rho(t,\cdot) \,||\, \rho_\infty) \leq 0.
$$
 Interpreting an evolutionary PDE as a gradient flow in the space
  of probability measures with respect to certain Wasserstein metric
  dates back to the seminar work on Fokker-Planck equation by Jordan,
  Kinderlehrer and Otto \cite{JKO}.  By now, similar gradient flow
  structures have been identified for a large family of evolution
  equations, including porous medium equation \cite{Otto01_porous},
  McKean-Vlasov equation \cite{carrillo2003kinetic}, etc.  In the
  present paper, we will not pursue further the rigorous definition
  and analysis of the gradient flow structure of
  \eqref{eq:mfl}. Instead, we take the system \eqref{eq:ps} as our
  starting point and prove its connection to the mean field PDE
  \eqref{eq:mfl}.  

\subsection{Relevant Literature}
Sampling from a density of the form $\rho_\infty(x) = e^{-V(x)}/Z$ without knowing the 
normalization constant $Z$ is a fundamental problem in Bayesian statistics and machine learning.
One generic approach that has been 
tremendously successful in recent years is the Markov chain Monte Carlo (MCMC)
methodology based on Metropolis-Hastings mechanism. 
The general principle of Metropolis-Hastings algorithms is 
to build an ergodic Markov chain 
whose invariant measure is the target measure $\rho_\infty$ by first 
making candidate samples (proposals), which are then tuned to ensure stationarity 
via acception/rejection. In practice, one common approach to
constructing proposals is by discretizing some stochastic dynamics, 
such as the following (overdamped) Langevin dynamics: 
\begin{align}
dX(t) = - \nabla V(X) \,dt + \sqrt{2}\, dB(t) \label{ito1},
\end{align}
where $B$ is a standard Brownian motion in $\R^d$. A vanilla Euler-Maruyama discretization scheme associated to \eqref{ito1} together with Metropolis-Hastings
step leads to the famous {\em Metropolis Adjusted Langevin Algorithm}
(MALA) \cite{roberts1996exponential, bou2012nonasymptotic} (whose
non-Metropolized version known as {\em unadjusted Langevin algorithm}
(ULA) \cite{dalalyan2017theoretical, durmus2017nonasymptotic}).
 
%
%A vanilla Euler-Maruyama discretization scheme associated to \eqref{ito1} leads to the so-called {\em unadjusted Langevin algorithm} (ULA):
%$$
%X_{k+1} = X_k - h\nabla V(X_k) + \sqrt{2h} \xi_k,
%$$
%where $\{\xi_k\}$ is an i.i.d. sequence of standard normal random variables in $\R^d$ and $h$ is the step size. Adding a Metropolis-Hastings step in each iteration above gives 
%the {\em Metropolis Adjusted Langevin Algorithm} (MALA).  
% In the limit $k \gt \infty$,
%these algorithms produce samples $X_k$ 
%whose law approximates $\rho_\infty$. Alternatively, one can 
%approximate $\rho_\infty$ by the ergodic occupation measure of one or multiple trajectories $\{X_k\}_{k\in \N}$. For a detailed  convergence analysis of ULA and MALA, we refer 
%interested readers to 
One advantageous feature of stochastic dynamics-based sampling
methods, e.g. MALA or ULA, is that the dynamics tend to explore high
probability regions (around the local minima of $V$), while the random
noise helps the dynamics to escape outside the basin of attraction and
thus promotes its exploration of the entire state space.  In contrast
to this stochastic sampling approach, \eqref{eq:ps} may be viewed as a
deterministic (albeit coupled) particle system for approximating $\rho_\infty$. Qualitatively speaking, the terms in \eqref{eq:ps} which
involve $\nabla V$ tend to drive particles toward local minima of $V$
(note however the nonlocal interaction due to the presence of $K$).
On the other hand, the terms involving $\nabla K$ are repulsive,
forcing the particles to disperse; this is seen in the fact that
\[
-\frac{1}{N} \sum_{j=1}^N \nabla K(x_i - x_j) = - \nabla_{x_i} E({\bf x}),
\]
where $E({\bf x}) = \frac{1}{N} \sum_{i < j} K(x_i - x_j)$ is the
interaction energy.  Here we assumed that $\nabla K(0) = 0$. This
interaction term in SVGD plays a role similar to that of the diffusion
term in stochastic-dynamics-based sampling methods.  Intuitively, one
would expect that the empirical measure $\mu^N_t$ of the particles
$\{x_i(t)\}$ tends to be close to $\rho_\infty$ in the limit of both
large sample size and long time.  One of the contributions of this
paper is to prove this convergence rigorously.

To compare these two sampling approaches at the PDE level, observe that the probability density for $X(t)$ defined by \eqref{ito1} solves the linear Fokker-Planck equation
\begin{align} \label{FP1}
\partial_t \rho = \nabla \cdot (\nabla \rho + \rho \nabla V). 
\end{align}
It is well known \cite{markowich2000trend} that under some mild assumption on $V$, the solution $\rho$
of \eqref{FP1} converges to the equilibrium distribution $\rho_\infty$ exponentially fast. 
On the other hand, if we formally set $K(x) = \delta_0(x)$, the non-local mean-field equation \eqref{eq:mfl} becomes\begin{align}  \label{Kdeltalimit}
\partial_t \rho = \nabla \cdot \left(\rho (\nabla \rho + \rho \nabla V) \right),
\end{align}
which is a non-linear porous medium equation with an additional
transport due to $\nabla V$. So, compared to \eqref{FP1}, the mobility
term and the transport term in \eqref{Kdeltalimit} are small where the
density is small. This suggests that the convergence of the solution
of \eqref{Kdeltalimit} towards $\rho_\infty$ may be slower than that
of \eqref{FP1}.  In this paper, we consider only a fixed kernel $K$,
but if we scale the kernel $K$ as
$K_N(\cdot) = N^\beta K (N^\beta \cdot)$, it is natural to expect the
large particle limit of \eqref{eq:ps} to be governed by
\eqref{Kdeltalimit} instead of \eqref{eq:mfl}, if $\beta > 0$ is not
too large --- rigorous justification of such convergence result is
still work in process.

 One should also compare \eqref{eq:ps} with the following more standard deterministic
 interacting particle system:
\begin{equation}\label{eq:ps2}
  \dot{x_i} = - \frac{1}{N} \sum_{j=1}^N \nabla K(x_i - x_j) - \nabla V(x_i), \qquad i = 1,2,\cdots, N.
\end{equation}
It is well-known \cite{dobrushin1979vlasov} that under suitable assumption
on $K$ and $V$, the mean field limit of \eqref{eq:ps2} is the following McKean-Vlasov equation
\be\label{eq:mv}
\partial_t \rho  = \nabla \cdot(\rho(\nabla( K \ast \rho + V))).
\en
The particle system \eqref{eq:ps} differs from
\eqref{eq:ps2}
in that the external force
added to each particle is non-local, and is defined
by averaging the individual forces $\nabla V(x_j)$ with 
weights defined by the kernel $K$. Interestingly, 
such non-local external force 
guarantees that $\rho_\infty$ is a stationary solution of \eqref{eq:mfl} ---
this explains the rationale for using the deterministic particle system 
\eqref{eq:ps} as an approximation algorithm for sampling $\rho_\infty$. 
On the contrary, $\rho_\infty$ is not a stationary solution of \eqref{eq:mv}. In fact, if $V$ or $K$ is non-convex, the equation
\eqref{eq:mv} may have multiple stationary solutions; see e.g. 
\cite{burger2008large, burger2013stationary}. We also remark that the nonlocal external force makes
the analysis of \eqref{eq:ps} more challenging than that of \eqref{eq:ps2}.

Although sampling via a deterministic particle system is less common,
the use of deterministic particles is ubiquitous in
numerical approximations of partial differential 
equations arising in physics and biology. For example, the point vortex method 
have been proved successful for solving equations in fluid mechanics 
\cite{goodman1990convergence, raviart1985analysis}, and similarly the weighted particle method
\cite{degond1989weighted} and the diffusion-velocity method \cite{degond1990deterministic} for
convection-diffusion and nonlinear-wave equations \cite{chertock2001particle}. 
For a comprehensive discussion on deterministic particle
methods
we refer the reader to the recent review paper \cite{chertock2017practical}
and references therein. Recently,
a blob method
was proposed in \cite{craig2016blob} 
for an aggregation equation,
which is the equation \eqref{eq:mfl} with $V = 0$ and with $K$ being attractive rather than repulsive. One typical aggregation equation
is the so-called Keller-Segel equation \cite{keller1970initiation, horstmann20031070}. 
The same blob method was generalized by \cite{CarilloCraigPatacchini17} to a
more general class of nonlinear diffusion equations, which
has a $L^2$-Wasserstein gradient flow structure. A key feature of the blob method considered there is that the particle system 
preserves a similar gradient flow structure as the diffusion equation, which facilitates the proof of large particle limits. On the contrary, the SVGD dynamics 
\eqref{eq:ps} is not a gradient flow. This again makes the analysis of the mean field limit non-trivial. 

\subsection{Plan of The Paper} The rest of the paper is organized as follows. In \cref{sec:pre}, 
we first make several technical assumptions on $V$ and $K$
 and then state our main results under these assumptions. 
  In \cref{sec:conv}, we prove the existence and uniqueness of weak solutions to the mean field equation \cref{eq:mfl} as well as the ODE system \cref{eq:ps} of SVGD by use of the mean field characteristic flow. Some useful estimates on the solution  \cref{eq:ps} are also derived. 
   \cref{sec:wellpose} concerns the regularity of the solution to the mean field equation \eqref{eq:mfl} under additional regularity assumption on $K$.
   \cref{sec:cps} devotes to the proof of the passage from the particles system \cref{eq:ps} to its mean field PDE 
  \cref{eq:mfl}. 
Finally, in \cref{sec:asymp}
 we prove that the solution 
 $\rho$ of \eqref{eq:mfl} converges to the equilibrium $\rho_\infty$ as $t \gt \infty$.

\section{Preliminaries and Main Results}\label{sec:pre}
\subsection{Assumptions and Notation}

Throughout the paper we assume that the kernel $K$ satisfies the following:

\begin{assumption}\label{ass:k}
 $K:\R^d\mapsto \R$ is at least $C^4$ with bounded derivatives.  In addition, $K(x-y)$ is symmetric and positive definite, meaning that
\[
 \sum_{i=1}^m\sum_{j=1}^m K(x_i- x_j)\xi_i\xi_j \geq 0, \quad \quad \forall \;\; x_i\in \R^d,\;\xi_i\in \R, \;\; m \in \mathbb{N}.
\]

\end{assumption}
A canonical choice of $K$ satisfying Assumption \ref{ass:k}
is a Gaussian kernel, e.g. $K(x) = \frac{1}{(4\pi)^{d/2}} \text{exp}(-\frac{|x|^2}{4})$. 
Higher regularity of $K$ will be needed to obtain higher regularity of the solution of the mean field PDE; see \cref{prop:inteq}. For the long time convergence of the solution, we will need further assumption on $K$; see \cref{thm:asymp}.

For the potential function $V: \R^d \mapsto \R$, we will assume the following:

\begin{assumption}\label{ass:v}

\begin{itemize}

 \item[(A1)]  $V\in C^\infty(\R^d), V\geq 0$ and $V(x) \gt +\infty$ if $|x| \gt +\infty$.
 
 \item[(A2)] There exists a constant $C_V > 0$ and some index $q > 1$ such that 
 $$
 |\nabla V(x)|^{q} \leq C_V(1 + V(x)) \text{ for every } x \in \R^d
 $$
 and that 
 \be\label{eq:v2}
\sup_{\theta \in [0,1]} |   \nabla^2 V(\theta x + (1-\theta)y)|^{q} \leq C_V(1 + V(x) + V(y)).
 \en
 
 \item[(A3)] For any $\alpha, \beta > 0$,  there exists a constant $C_{\alpha, \beta}>0$ such that if $|y| \leq \alpha |x| + \beta$, then
 $$
 (1 + |x|)(|\nabla V(y)| + |\nabla^2 V(y)|) \leq C_{\alpha, \beta} (1 + V(x)).
 $$

\end{itemize}

\end{assumption}

\begin{remark}
 We comment that  \cref{ass:v} (A1)-(A3) will be used in the proofs of the existence, uniqueness
 and regularity of the solution of mean field equation. Note that by setting $\alpha = 1, \beta  = 0$ and $y = x$ in (A3), we have that 
\be\label{eq:asv3}
 (1 + |x|)(|\nabla V(x)| + |\nabla^2 V(x)|)  \leq C_{1} (1 + V(x))
\en
for some constant $C_1 > 0$. 
 These assumptions are by no means sharp, but proves to be sufficient
 for the validity of our theorems.  \cref{ass:v} (A2) implies that 
there is $C_0$ such that 
\begin{equation}
V(x) \leq C_0(1 + |x|^{q^\ast}) \quad \forall \;\; x \in \R^d \label{vpgrowth}
\end{equation}
where $ q^*  = \frac{q}{q-1}$. Indeed, this follows from
\[
\frac{d}{dt} \left( 1 + V(t \hat n)\right)^{\frac{q - 1}{q}} = (\frac{q - 1}{q}) \frac{ \hat n \cdot \nabla V(t \hat n)}{ (1 + V(t \hat n))^{1/q}} \leq (1 - \frac{1}{q}) C_V^{1/q}.
\]
where $\hat n = x/|x|$, and then integrating from $t = 0$ to $t = |x|$.  It is also easy to check that 
 \cref{ass:v} is fulfilled by even polynomials up to order $q^*$. 
\end{remark}

We use $\P_V$ and $\P_p$ denote the set of Borel probability measures $\mu$ on $\R^d$ satisfying
\begin{align} \label{Pmeasspace}
\| \mu \|_{\P_V} = \int_{\R^d} (1 + V(x)) \,d\mu < \infty \quad \quad \text{or} \quad \quad \| \mu \|_{\P_p} = \int_{\R^d} |x|^p \,d\mu(x) < \infty,
\end{align}
respectively. Thanks to \cref{vpgrowth}, we have $\P_p \subset \P_V$ for any $p \geq q^\ast = \frac{q}{q-1}$. For $\mu, \nu \in \P_p$, $\W_{p}(\mu, \nu)$ denotes the $p$-Wasserstein distance \cite{Vil03}.  Given a probability measure $\mu$ and a Borel-measurable map $f$, we denote by $f_\# \mu$ the push-forward of the measure $\mu$ under the map $f$. In places where $\rho$ is time-dependent, we often use notation $\rho_t = \rho(t,\cdot)$ to emphasize this time dependence in a succinct way; on the other hand, differentiation with respect to the variable $t$ will always be denoted by $\partial_t \rho$.

%%%%%%%%%%%%%%%%%%%%
%\commentout{
For $k,p\geq 1$, we denote by $W^{k, p}(\R^d)$ the usual Sobolev space of functions whose weak derivatives up to $k$-th order belong to $L^p(\R^d)$. When $p=2$, we write $H^k(\R^d) = W^{k,2}(\R^d)$. For our result on regularity of solutions to the PDE \eqref{eq:mfl}, we introduce function spaces
$$
\bea
& L^1_{V}:= \{u\in L^1(\R^d)\ |\  \int_{\R^d} (1 + V(x))|u(x)|dx < \infty \}, \\
& W^{1,1}_V := \{u\in W^{1,1}(\R^d)\ |\  \int_{\R^d} (1 + V(x))(|u(x)| + |\nabla u(x)|)dx < \infty \}
\ena
$$
with norms $\|u\|_{L^1_{V}} := \|(1 + V) u\|_{L^1(\R^d)}$ and 
$
\|u\|_{W^{1,1}_V } := \int_{\R^d}(1 + V(x))(|u(x)| + |\nabla u(x)|)dx.
$
respectively. We set 
$$
\sY_{k, V} := H^k(\R^n) \cap L^1_{V}, \quad \quad \sY_{k, V}^1 = H^k(\R^n) \cap W^{1,1}_V
$$
with the canonical norms 
$$
\|u\|_{\sY_{k, V}} := \|u\|_{H^k(\R^d)} + \|u\|_{L^1_V}, \quad \quad \|u\|_{\sY_{k, V}^1} := \|u\|_{H^k(\R^d)} + \|u\|_{W^{1,1}_V}.
$$

%}%%%%%%%%%%%%%%%%%end commentout

We will use constant $C(V)$ to denote a generic constant which depends on $V$. Similar rules apply to $C(K), C(V, K)$, etc.  We also use constants $C, \tilde{C}, \tilde{\tilde{C}}$  to denote generic constants that are independent of quantities of interest. The exact values of these constants may change from line to line.

\subsection{Main Results}

Our first result is the global well-posedness of the nonlinear mean field PDE \eqref{eq:mfl}.  Observe that \eqref{eq:mfl} is a nonlinear transport equation of the form $\partial_t \rho + \nabla \cdot (\rho U[\rho]) = 0$, where $U[\rho]$ is the vector field
\begin{equation}
U[\rho](x) = - (\nabla K * \rho)(x) -(K * (\nabla V \rho))(x), \quad x \in \R^d. \label{Uvecdef}
\end{equation}
Given a measure $\rho \in \P_V$, $U[\rho]$ is well-defined. In fact, due to Assumption \ref{ass:v} (A2), $U[\rho](x)$ is Lipschitz continuous and bounded over $\R^d$:
\begin{align} \label{Uunifbound}
 |U[\rho](x)| & \leq \| \nabla K \|_\infty + \| K \|_{\infty} C \| \rho \|_{\P_V},  \\
|\nabla U[\rho](x)| & \leq \| D^2 K \|_\infty + \| D K \|_{\infty} C \| \rho \|_{\P_V}. \nonumber
\end{align} 
We say that a measure-valued function $\rho \in C([0,\infty);\P)$ (where $\P$ is given the topology of weak convergence) is a weak solution to \eqref{eq:mfl} with initial condition $\rho_0 = \nu \in \P_V$ if
\begin{equation}
\sup_{t \in [0,T]} \| \rho_t\|_{\P_V} < \infty, \quad \forall \;T > 0 \label{weakMoment1}
\end{equation}
and
\[
\int_0^\infty \int_{\R^d} \left( \partial_t \phi(t,x) + \nabla \phi(t,x) \cdot U[\rho_t](x) \right) \rho_t(dx) \,dt + \int_{\R^d} \phi(0,x) \nu(dx) = 0
\]
holds for all $\phi \in C^\infty_0([0,\infty) \times \R^d)$. Recall that $\rho_t = \rho(t,\cdot)$.

\begin{theorem}\label{thm:weak}
Let $V$ satisfy \cref{ass:v}. For any $\nu \in \P_V$, there is a unique $\rho \in C([0,\infty); \P_V)$ which is a weak solution to \eqref{eq:mfl} with initial condition $\rho_0 = \nu$. Moreover there is $C_1 > 0$ (depending on $K$ and $V$) such that
\begin{equation}
\| \rho_t \|_{\P_V} \leq e^{C_1 t} \|\nu \|_{\P_V}, \quad t \geq 0. \label{expVbound}
\end{equation}
If $\nu \in \P_p \cap \P_V$, then $\rho \in C([0,\infty); \P_p)$, as well, with $\| \rho_t\|_{\P_p} \leq e^{C_2 t}  \| \nu \|_{\P_p}$.
\end{theorem}

The theorem is proved in Section~\ref{sec:mfcf}
Our next result, proved in Section~\ref{sec:estimateps},  establishes that the finite particle system is well-posed, and that the associated empirical measure is a weak solution of the PDE \eqref{eq:mfl}:

\begin{proposition}\label{prop:existence-ps}
 Let $V$ satisfy \cref{ass:v}. Then for any initial condition 
 $\mathbf{x}^{0} = \{x_i^{0}\}_{i=1}^N \in \R^{dN}$, the system \eqref{eq:ps} has a unique global solution $\mathbf{x}(t) = \{x_i(t)\}_{i=1}^N\in C^1([0, \infty); \R^{dN})$, and the measure $\mu^N_t = \frac{1}{N} \sum_{i=1}^N \delta_{x_i(t)}$ is a weak solution to the PDE \eqref{eq:mfl}. 
\end{proposition}

In particular, the bound \eqref{expVbound} holds for the empirical measure $\mu_t^N$. With additional assumptions about the behavior of $V$ and $K$ as $|x| \to \infty$, we are able to improve upon \eqref{expVbound} and show that $\| \mu_t^N \|_{\P_V}$ is bounded in time; see Lemma \ref{lem:bddintime} below.

When the initial condition is more regular, then the weak solution inherits higher regularity, as described by the following proposition. We remark that this regularity result will not be used in our proof of the mean field limit,  but is of interest on its own account from the PDE perspective. 

\begin{proposition}\label{prop:inteq}
  Let $V$ satisfy \cref{ass:v}. Suppose that $\nu \in \P_V$ has a
  density $\rho_0(x) \geq 0$. If $\rho_t$ is the unique weak solution
  to \eqref{eq:mfl} with this initial condition, then $\rho_t$ also
  has a density. Furthermore, if $\rho_0 \in \sY^1_{k, V}$ for some
  $k \geq 2$, and the kernel $K$ is $k+2$ times differentiable with
  bounded derivatives, then $\rho_t$ has a density satisfying
  \be\label{eq:apriori1} \|\rho_t(\cdot)\|_{\sY^1_{k, V}} \leq
  \exp\Big( C_2 (e^{ C_1 t} - 1)\lVert \rho_0 \rVert_{\P_V} \Big)
  \|\rho_0\|_{\sY^1_{k, V}}, \quad t \geq 0 \en where the constants
  $C_1, C_2$ depend only on $V$ and $K$.
\end{proposition}

In the case that $V = 0$, a similar regularity result to \cref{eq:apriori1}  was proved for aggregation equation by Laurent \cite{laurent2007local}.  The presence of
the potential $V$ makes the problem more difficult since the velocity $\nabla V$ is unbounded at infinity. This difficulty was circumvented with the help of the mean field characteristic flow (c.f. \cref{def:mfcf}), which allows us to express the solution $\rho_t$ in terms of the initial condition and the flow map. The regularity of $\rho_t$ simply transfers from that of $K$ provided we can show $\rho(t,\cdot) \in W^{1,1 }_V$. See the detailed proof in \cref{sec:wellpose}.

Next, we prove a stability estimate for weak solutions to \eqref{eq:mfl}.

\begin{theorem}\label{thm:stability}
Let $V$ satisfy \cref{ass:v} with $q \in (1,\infty)$ in (A2). 
Let $p$ be the conjugate index of $q$, i.e. $\frac{1}{p} +\frac{1}{q} = 1$.
Let $R > 0$. Assume that $\nu_1, \nu_2$ are two initial probability measures in $\P_p$ satisfying $\| \nu_i \|_{\P_p} \leq R$, $i = 1,2$.  Let $\mu_{1,t}$ and $\mu_{2,t}$ be the associated weak solutions to \eqref{eq:mfl}.  Then given any $T>0$, there exists a constant $C > 0$ depending on $ K, V, R, p$ and $T$ such that 
\be\label{eq:stability}
\sup_{t \in [0,T]} \W_p(\mu_{1,t}, \mu_{2, t}) \leq C \W_p(\nu_1, \nu_2).
\en
\end{theorem}

Theorem \ref{thm:stability} addresses the behavior of the particle system as $N \to \infty$.  Suppose that the initial points
$\{x_i^N(0)\}_{i=1}^N$ are such that $\W_p(\mu^N_0 , \nu_0) \to 0$ as
$N \to \infty$. Then if $\rho_t$ is the unique weak solution to \eqref{eq:mfl} with initial condition $\nu_0$, Theorem \ref{thm:stability} implies that $\W_p(\mu^N_t , \rho_t) \to 0$ uniformly over $[0,T]$, since $\mu^N_t$ is a weak solution to \eqref{eq:mfl}. This hypothesis of the convergence of the
initial empirical measure, i.e., $\W_{p}(\mu^N_0, \nu_0) \gt 0$, can
be justified rigorously, e.g.,  when the initial particles
$\{x^0_i\}$ are independent samples drawn from $\nu_0$.  For a
detailed discussion on the convergence of empirical measures in
$\W_p$, we refer the interested readers to references
\cite{fournier2015rate,weed2017sharp,bobkov2014one, trillos2014rate,LLL18}.

We prove \cref{thm:stability} in Section
\ref{sec:cps} by following Dobrushin's coupling argument 
\cite{dobrushin1979vlasov,golse2016dynamics} for the mean
field characteristic flow (defined later at \eqref{eq:mfcf}).  The proof follows closely the proof of
Theorem 1.4.1 of \cite{golse2016dynamics}, which dealt with the case $V = 0$. The stability estimate there was stated in terms of
$1$-Wasserstein distance, and mainly resulted from the Lipschitz condition of $\nabla K$.
However, we 
are only be able to prove the stability of mean field characteristic flow in $p$-Wasserstein distance with $p$ strictly larger than one. This is again due to the presence
of the nonlinear drift term $K\ast (\nabla V \rho)$ in the vector field \eqref{Uvecdef}.

Our last result pertains to the long time behavior of solutions $\rho_t$ of \eqref{eq:mfl} with sufficiently regular initial condition. Since the probability density $\rho_\infty(x) = e^{-V(x)}/Z$ is an invariant solution to the PDE \eqref{eq:mfl}, it is natural to ask whether $\rho_\infty$ is the unique invariant measure, and whether $\rho_t \to \rho_{\infty}$ as $t \to \infty$. Generally speaking, \eqref{eq:mfl} may admit many invariant measures.  For example, for any stationary solution to the finite particle system \eqref{eq:ps}, the empirical measure $\mu^N$ corresponds to a (stationary) weak solution of the PDE \eqref{eq:mfl}; there may be many such stationary solutions. However, if we restrict to initial conditions $\rho_0$ which are absolutely continuous with respect to $\rho_\infty$, one may expect that solutions to \eqref{eq:mfl} converge to $\rho_\infty$ as $t \to \infty$. The following theorem  confirms this intuition.   For technical reasons, 
we need to make further assumptions on the kernel $K$.

\begin{theorem}\label{thm:asymp}
 Let $V$ satisfy \cref{ass:v}. Assume that $K$ satisfies \cref{ass:k} and the following extra assumption:
 \be\label{eq:extraK}
K = K_{1/2}\ast K_{1/2}  \text{ with } K_{1/2}\in \mathcal{S} \text{ and } \hat{K}_{1/2}(\xi) > 0,\ \forall \xi \in \R^d.
\en
    Let $\rho_t$ be the solution to \eqref{eq:mfl} with initial
  condition $\rho_0\in \mathscr{Y}^1_{2,V}$ satisfying
  $\KL(\rho_0 \,||\, \rho_\infty) < \infty$.  Then $\rho_t$ converges
  weakly to $\rho_\infty$ as $t \gt \infty$.
\end{theorem}

\cref{thm:asymp} in particular implies that $\rho_\infty = e^{-V}/Z$
is the unique equilibrium of the mean field equation \cref{eq:mfl}
provided that the initial distribution $\rho_0$ has a density and
satisfies $\text{KL}(\rho_0 || \rho_\infty) <\infty$. However, if the
initial distribution is discrete, such as in the case of the particle
system \eqref{eq:ps}, there could be multiple equilibria, in which
case the long time behavior of $\rho_t$ may depend on the initial
distribution.

The proof of Theorem \ref{thm:asymp} is presented in \cref{sec:asymp}. 
% Extension of the result to other smooth reproducing kernels is unclear to us and remains an interesting open problem. Moreover, 
A quantitative convergence rate is far from clear to us. The main
obstacle is the lack of a generalized logarithmic Sobolev inequality
which could lower bound the Stein discrepancy in terms of the relative
entropy. This issue is to be investigated in future works.  Another
important unresolved issue is whether ``generic'' stationary solutions
of the particle system \eqref{eq:ps} are close in some sense to
$\rho_\infty$, when $N$ is large.

\section{Well-posedness of the PDE and the particle system}\label{sec:conv}

In this section we prove Theorem \ref{thm:weak} and Proposition \ref{prop:existence-ps}. The main ingredient in the proof of Theorem \ref{thm:weak} is the so-called mean field characteristic flow, introduced in \cref{sec:mfcf}.  In \cref{sec:estimateps}, we prove Proposition \ref{prop:existence-ps} and an additional estimate on the particle system under strong assumptions on $V$.

\subsection{Mean field characteristic flow}\label{sec:mfcf}
Here we define the {\em mean field characteristic flow} for the PDE \eqref{eq:mfl} (c.f \cite{golse2016dynamics}), which will play an essential role in the proof of the large particle limit of \eqref{eq:ps}.

\begin{definition} \label{def:mfcf}
Given a probability measure $\nu$, we say that the map
$$
X(t, x, \nu): [0,\infty) \times \R^d \to \R^d
$$
is a mean field characteristic flow associated to the particle system \eqref{eq:ps} or to the mean field PDE \eqref{eq:mfl} if $X$ is $C^1$ in time and solves the following problem
 \be\label{eq:mfcf}
 \bea
&  \partial_t X(t,x, \nu)  = - (\nabla K \ast\mu_t)(X(t,x, \nu) ) - (K\ast (\nabla V \mu_t))(X(t,x, \nu)), \\
 &  \mu_t = X(t, \cdot, \nu)_{\#} \nu,\\
 & X(0, x, \nu) = x.
 \ena
 \en
\end{definition}

The expression $ \mu_t = X(t, \cdot, \nu)_{\#} \nu$ means that the measure $\mu_t$ is the push-forward of $\nu$ under the map $x \mapsto X(t,\cdot,\nu)$.  We think of $\{X(t,\cdot,\nu)\}_{t \geq 0, \nu}$ as a family of maps from $\R^d$ to $\R^d$, parameterized by $t$ and $\nu$. We first prove in the theorem below that the mean field characteristic flow \eqref{eq:mfcf}
is well-defined. To this end, define the set of functions
$$
Y  := \Big\{u\in C(\R^d; \R^d)\ |\  \sup_{x\in \R^d} |u(x) - x| < \infty \Big\},
$$
which is a complete metric space with the uniform metric $d_Y(u,v) = \sup_{x} |u(x) - v(x)|$. Recall the space of measures $\P_V$ defined in \eqref{Pmeasspace}.

\begin{theorem} \label{theo:mfcf}
Assume the conditions of Theorem \ref{thm:weak} hold, and $\nu \in \P_V$. For any $T>0$,  there exists a unique solution $X(\cdot,\cdot, \nu)\in C^1([0, T]; Y)$ to the problem \eqref{eq:mfcf}. Moreover, the measure $\mu_t = X(t, \cdot, \nu)_{\#} \nu$ satisfies $\| \mu_t\|_{\P_V} \leq e^{C \|\nabla K\|_{\infty} t }\| \nu\|_{\P_V}$, some constant $C$ that is independent of $\nu$. 
\end{theorem}
\begin{proof}
We follow the proof of Theorem 1.3.2 in \cite{golse2016dynamics}. The proof of the  theorem consists of  two steps.  

\noindent{\bf Step 1 (local well-posedness):} Fix $r > 0$, and define
\[
Y_r  := \Big\{u \in Y\ |\  \sup_{x\in \R^d} |u(x) - x| \leq r\Big\}.
\]
We prove that there exists $T_0 > 0$ such that the problem \eqref{eq:mfcf} has a unique solution $X(t, x)$ in the set
\[
S_r = C([0, T_0];Y_r)
\]
which is a complete metric space, with metric
\[
d_{S}(u,v) = \sup_{t \in [0,T_0]} d_Y(u(t,\cdot),v(t,\cdot)).
\]
Consider the integral formulation of \eqref{eq:mfcf} given by 
 \be\label{eq:intX}
\bea
& X(t, x, \nu) = x - \int_0^t  \int_{\R^d} \nabla K(X(s, x, \nu) - X(s, x^\prime, \nu)) \nu(dx^\prime)ds\\
& -\int_0^t\int_{\R^d}  K(X(s, x, \nu) - X(s, x^\prime, \nu)) \nabla V(X(s, x^\prime, \nu))\nu(dx^\prime) ds. 
\ena
\en
Let us define the operator $\F : u(t,\cdot) \mapsto \F(u)(t,\cdot)$ by 
$$\bea
\F(u)(t,x) & := x - \int_0^t  \int_{\R^d} \nabla K(u(s,x) - u(s, x^\prime)) \nu(dx^\prime)ds\\
& -\int_0^t\int_{\R^d}  K(u(s, x) - u(s, x^\prime)) \nabla V(u(s, x^\prime))\nu(dx^\prime) ds. 
\ena
$$
Our goal is to show that $\F$ is a contraction in $S_r$, and thus has a unique fixed point.

We first show that $\F$ maps $S_r$ into $S_r$. Checking that $(t,x) \mapsto \F(u)(t,x)$ is continuous is straightforward; we need to establish a bound on $|\F(u)(t,x) - x|$.  If $u\in S_r$, then for any $s\in [0, T_0]$ and $x^\prime \in \R$, 
\be\label{eq:bdu}
|u(s, x^\prime)| \leq |x^\prime| + |u(s, x^\prime) - x^\prime| \leq  |x^\prime| + r.
\en
Then according to Assumptions \eqref{ass:v} (A3), there exists a positive constant $C_r$ such that 
\be\label{eq:v1}
|\nabla V(u(s, x^\prime))| \leq C_r(1 + V(x^\prime)), \quad \forall\;\; x^\prime \in \R^d.
\en
As a consequence, we have 
$$
|\F(u)(t, x) - x| \leq  t \| \nabla K \|_\infty + t C \|K\|_{\infty}\int_{\R^d} (1 + V(x^\prime))\nu(dx^\prime) \leq \tilde{C} t,
$$
where we used the assumption that $\nu \in \P_V$. Therefore,
$$
 \sup_{t\in [0, T_0]} \sup_{x \in \R^d} |\F(u)(t, x)  - x | \leq \tilde{C}T_0 \leq r
$$
if $T_0  \leq r/\tilde{C}$. This shows that $\F$ maps from $S_r$ to $S_r$, if $T_0$ is sufficiently small. 

Next, we show that $\F$ is indeed a contraction on $S_r$. If $u, v \in S_r$, then for any $t \in [0,T_0]$ and $x \in \R^d$,
\begin{align}
& |\F(u)(t, x)  -\F(v)(t, x) | \label{eq:difF}\\
 & \leq \int_0^{T_0} \Big | \int_{\R^d} \nabla K(u(s,x) - u(s, x^\prime)) - \nabla K(v(s,x) - v(s, x^\prime)) \nu(dx^\prime) \Big| ds \nonumber \\
& + \int_0^{T_0}\Big |\int_{\R^d} \Big( K(u(s, x) - u(s, x^\prime)) -  K(v(s,x) - v(s, x^\prime))\Big) \nabla V(u(s, x^\prime))\nu(dx^\prime)\Big | ds \nonumber \\
& + \int_0^{T_0}\Big |\int_{\R^d} K(v(s,x) - v(s, x^\prime)) \Big(\nabla V(u(s, x^\prime)) - \nabla V(v(s, x^\prime))\Big)  \nu(dx^\prime)\Big | ds. \nonumber
\end{align}
The first term on the right side above can be bounded from above by
\[
 T_0 \|D^2 K\|_{\infty}  2 d_S(u , v).
\]
Thanks to \eqref{eq:bdu} and \eqref{eq:asv3}, the second term can be bounded from above by
\begin{align}
 T_0  \|\nabla K\|_{\infty} 2 d_S(u,v) \int_{\R^d}  C_r (1 +  V(x^\prime))\nu(dx^\prime).
\end{align}
To bound the last term on the right side of \eqref{eq:difF}, using \cref{ass:v} (A2) one obtains that
\begin{align}
\Big|\nabla V(u(s, x^\prime)) - \nabla V(v(s, x^\prime))\Big| & \leq \max_{\theta\in[0,1]}|\nabla^2 V(\theta u(s, x^\prime) + (1- \theta)v(s, x^\prime)) |  d_S( u, v)  \nonumber \\
& \leq C_r(1+ V(x^\prime)) d_S( u, v),  \label{eq:diffdF}
\end{align}
where in the last inequality we have used the fact that $\theta u + (1 -\theta)v \in S_r$ so that $\theta u + (1 -\theta)v$ also satisfies the inequality \eqref{eq:bdu}, which enables
us to apply (A3) of \cref{ass:v}. Plugging \eqref{eq:diffdF} into the integral of the last term on the right side of \eqref{eq:difF}, we can bound the last term by
$$
T_0 C_r \|K\|_{\infty} \int_{\R^d} (1 + V(x^\prime))\nu(dx^\prime) d_S(u,v).
$$
Combining the estimates above leads to 
$$
d_S( \F(u), \F(v))  \leq T_0 \left( 2 \|D^2 K\|_{\infty}  + 2 C_r \left(\|K\|_{\infty} + \|\nabla K\|_{\infty}\right) \int_{\R^d} (1 + V(x^\prime))\nu(dx^\prime) \right)  \,d_S(u,v).
$$
which implies that $\F$ is a contraction on $S_r$ when $T_0$ is small enough.
By the contraction mapping theorem, $\F$ has a unique fixed point $X(\cdot, \cdot, \nu) \in S_r$, 
which solves \eqref{eq:intX}.
After defining $\mu_t = X(t, \cdot, \nu)_{\#} \nu$, one sees that $X(t, x, \nu)$ solves \eqref{eq:mfcf} in the small time interval $[0,T_0]$.  

\smallskip 
\noindent{\bf Step 2 (Extension of local solution): } Considering the bounds in the previous step, it is clear that the local solution may be extended beyond time $T_0$ as long as the quantity
\[
\| \mu_t \|_{\P_V} =  \int_{\R^d} \left( 1 + V(X(t, x, \nu)) \right) \nu(dx)
\]
remains finite.  We now establish an a priori bound on this quantity, showing that the local solution may be extended for all $t > 0$. 
\begin{align}
& \partial_t \int_{\R^d} \Big( 1 + V(X(t, x, \nu))\Big) \nu(dx) \\
& = - \int_{\R^d} \int_{\R^d} \nabla V(X(t, x, \nu)\cdot \nabla K(X(t, x, \nu) - X(t, x^\prime, \nu)) \nu(dx^\prime)\nu(dx) \nonumber \\
&  - \int_{\R^d} \int_{\R^d}K(X(t, x, \nu) - X(t, x^\prime, \nu))\nabla V(X(t, x, \nu)\cdot\nabla V(X(t, x^\prime, \nu)\nu(dx^\prime)\nu(dx) \nonumber \\
& \leq \|\nabla K\|_{\infty} \int_{\R^d} |\nabla V(X(t, x, \nu))| \nu(dx). \nonumber \\
& \leq C_r \|K\|_{1, \infty} \int_{\R^d} \Big( 1 + V(X(t, x, \nu))\Big) \nu(dx). \nonumber
\end{align}
The last inequality follows from  \cref{ass:v} (A3) and
the fact that $K$ is positive definite so
that the third line above is non-positive.
As a consequence, 
\be\label{eq:vxbd}
\int_{\R^d} \Big( 1 + V(X(t, x, \nu))\Big) \nu(dx) \leq e^{C_r \|\nabla K\|_{\infty} t }\int_{\R^d} \Big( 1 + V(x)\Big) \nu(dx)
\en
holds for all $t \in [0,T_0]$. With bound, one can iterate the argument to extend the local solution defined on $[0,T_0] \times \R^d$ to all of $[0,\infty) \times \R^d$, so that $\| \mu_t\|_{\P_V} \leq e^{C_r \|\nabla K\|_{\infty} t }\| \nu\|_{\P_V}$ holds for all $t > 0$. Similarly, there is $C > 0$ (depending on $r$) such that $d_Y(X(t,x,\nu), x) \leq C e^{C t}$ holds for all $t \geq 0$. Finally, thanks to the integral formulation \eqref{eq:intX} $\partial_t X$ is continuous on $[0,\infty) \times \R^d$. The proof is complete.  
\end{proof}

\begin{proof}[Proof of Theorem \ref{thm:weak}]
Given $\nu$, let $X(t,x,\nu)$ be the mean field characteristic flow defined in Theorem \ref{theo:mfcf}, and let $\rho_t = X(t,\cdot,\nu)_\# \nu$.   Then this $\rho_t$ is a weak solution to \eqref{eq:mfl} in the sense described above -- this follows immediately from Theorem 5.34 in \cite{Vil03}, for example.

Suppose that $\nu \in \P_p\cap \P_V$. As shown in the proof of Theorem \ref{theo:mfcf}, the map $X(t,x,\nu)$ is an element of the space $Y$ with $d_Y(X(t,x,\nu), x) \leq Ce^{C t}$.  Therefore, since $|X(t,x,\nu)|^p \leq 2^p |x|^p + 2^p d_Y(X(t,x,\nu),x))^p$, we have
\begin{align}
\int_{\R^d} |y|^p \rho_t(dy) & = \int_{\R^d} |X(t,x,\nu)|^p \nu(dx) \leq C \int_{\R^d} |x|^p \nu(dx) + C e^{C t}
\end{align}
for all $t > 0$. Hence $\rho_t \in \P_p \cap \P_V$ for all $t > 0$. 

Uniqueness of the weak solution follows from uniqueness of the mean field characteristic flow, as we now explain. Suppose $q \in C([0,T];\P)$ is any other weak solution. Because $q$ satisfies \eqref{weakMoment1}, the vector field $(t,x) \mapsto U[q_t](x)$ is bounded over $[0,T] \times \R^d$, continuous in $t$ and Lipschitz continuous in $x$. Then we can define a continuous family of maps $\tilde X(t,\cdot,\nu)$ by
\begin{align}
& \partial_t \tilde X = U[q_t](\tilde X) \\
& \tilde X(0,x,\nu) = x.
\end{align}
Because of \eqref{weakMoment1}, the measure $\tilde q_t = \tilde X(t,\cdot,\nu)_\# \nu$ is a weak solution to the linear transport equation
\[
\partial_t \tilde q + \nabla \cdot ( \tilde q U[q_t](x)) = 0
\]
with initial condition $\tilde q_0 = \nu = q_0$.  Uniqueness, of the solution to this linear equation implies that $\tilde q_t = q_t$.  That is, $ \tilde X(t,\cdot,\nu)_\# \nu = q_t$, which means that $\tilde X(t,x,\nu)$ is the mean field characteristic flow for $\nu$.  Uniqueness of the mean-field characteristic flow implies that $\tilde X = X$, hence $q_t = \rho_t$. This proves that the weak solution is unique.
\end{proof}

\subsection{Estimates on the particle system}\label{sec:estimateps}

In this section, we prove Proposition \ref{prop:existence-ps}, showing that the particle system \eqref{eq:ps} is well-posed and that the empirical measure is a weak solution to the mean field PDE. It is useful to introduce the function
$$
H_N(\mathbf{x}) = \frac{1}{N} \sum_{i=1}^N V (x_i) + 1,
$$ 
where $\mathbf{x} = (x_1, x_2, \cdots, x_n)$.

\begin{proof}[Proof of Proposition \ref{prop:existence-ps}]
Since both $K$ and $V$ are $C^2$, it is well-known that the problem \eqref{eq:ps} has a unique solution up to some time $T_0 > 0$. So, we must show that the solution does not blow up at a finite time. We claim that for some constant $C$,
\begin{equation}\label{eq:dbdV}
H_N (\mathbf{x}(t)) \leq H_N (\bx^0)\cdot e^{Ct}.
\end{equation}
This estimate and \cref{ass:v} (A1) imply that 
 $x_i(t)$ remains bounded over $[0, T]$ for any $T > 0$, whence 
  the solution can be extended up to any finite time.  To establish \eqref{eq:dbdV}, we first differentiate
  $V(x_i(t))$ with respect to $t$ and sum over $i$:
 $$\bea
 \partial_t\Big(\frac{1}{N} \sum_{i=1}^N V(x_i(t))\Big) & = -\frac{1}{N^2}\sum_{i,j=1}^N  \nabla K(x_i(t) - x_j(t))\cdot \nabla V(x_i(t)) \\
 & \qquad -\frac{1}{N^2}\sum_{i,j=1}^N K(x_i(t) - x_j(t))  \nabla V(x_i(t)) \cdot \nabla V(x_j(t)).
 \ena$$
 Observe that the second term on the right side of above is non-positive
 since the matrix $\{K(x_i- x_j)\}_{i,j=1}^N$ is positive definite 
 by Assumption \eqref{ass:k}.
 Then it follows from the inequality in Assumptions \eqref{ass:v} (A-2) and the fact that $\nabla K$ is uniformly bounded that there exists a constant $C = C(V, K) > 0$ such that
 $$
 \bea
  \partial_t\Big(\frac{1}{N} \sum_{i=1}^N V(x_i(t))\Big) & \leq \Big|\frac{1}{N^2}\sum_{i,j=1}^N  \nabla K(x_i(t) - x_j(t))\cdot \nabla V(x_i(t))\Big|\\
  & \leq  \frac{C}{N}\sum_{i=1}^N  |\nabla V(x_i(t))| \\
  & \leq C (\frac{1}{N} \sum_{i=1}^N V(x_i(t)) + 1).
 \ena
 $$
 This proves \eqref{eq:dbdV}.

Now having established well-posedness of the finite particle system, it now follows from the definition of the mean field characteristic flow $X(t,x,\mu^N_0)$ that 
\[
x_i(t) = X(t, x_i^0, \mu^N_0)
\]
and
\begin{align} \label{rem:psmfl}
\mu^N_t(dx) = (X(t, \cdot, \mu^N_0))_{\#} \mu^N_0.
\end{align}
In view of the proof of Theorem \ref{thm:weak}, we conclude that $\mu^N_t$ is a weak solution to the mean field PDE \eqref{eq:mfl}.

\end{proof}

The estimate \eqref{eq:dbdV} can be regarded as a discrete analogue of the estimate \eqref{expVbound} established in Theorem \ref{thm:weak}.  We expect that for fixed $N$, $H_N(\bx)$ will remain uniformly bounded in time, although we have been able to prove this only with some further restrictions on $V$ and $K$, as the next lemma states.

\begin{lemma}\label{lem:bddintime}
Fix $N \geq 1$. Suppose that for some $p \geq 2$ and $m, R > 0$, $V(x) = m |x|^p$ if $|x| > R$.  Suppose also that $K(0) > 0$ and that $|x|^{p-1}K(x)$ is bounded. Then $H_N(x(t))$ is uniformly bounded for $t \in [0,\infty)$. 
\end{lemma}
\begin{proof}
Observe that $\partial_t H_N(x(t)) = -(S_1 + S_2)/N^2$, where
\begin{align}
S_1 & =\sum_{i,j=1}^N  \nabla K(x_i(t) - x_j(t))\cdot \nabla V(x_i(t)) \nonumber, \\
S_2 & = \sum_{i,j=1}^N K(x_i(t) - x_j(t))  \nabla V(x_i(t)) \cdot \nabla V(x_j(t)). \nonumber 
\end{align}
Because $K$ is positive definite, we know that $S_2 \geq 0$. We wish to bound $S_2$ from below.   For $y \in \R^d$, let us define
\[
A(y) = \left \{ j \in \{1,\dots,N\} \;|\; \quad \nabla V(y) \cdot \nabla V(x_j) \geq |V(y)|^2/2 \right \}.
\]
We write $S_2$ as
\begin{align}
S_2 & = \sum_{i=1}^N \sum_{j \in A(x_i)} K(x_i(t) - x_j(t))  \nabla V(x_i(t))\cdot \nabla V(x_j(t)) \nonumber  \\ 
& \qquad \qquad + \sum_{i=1}^N \sum_{j \notin A(x_i)} K(x_i(t) - x_j(t))  \nabla V(x_i(t))\cdot \nabla V(x_j(t))  \nonumber  \\ 
& \geq   \sum_{i=1}^N \frac{1}{2} | \nabla V(x_i(t))|^2  \sum_{j \in A(x_i)} K(x_i(t) - x_j(t)) \nonumber  \\
  & \qquad \qquad + \sum_{i=1}^N \sum_{j \notin A(x_i)} K(x_i(t) - x_j(t))  \nabla V(x_i(t)) \cdot (\nabla V(x_j) - \nabla V(x_i(t))). \nonumber
\end{align}
Since
\[
|\nabla V(x_j) - \nabla V(x_i)| \leq  \int_0^1 \norm{D^2 V(s x_j + (1 - s)x_i)}\,ds |x_j - x_i|,
\]
we have
\begin{align}
S_2 & \geq  \sum_{i=1}^N \frac{1}{2} | \nabla V(x_i(t))|^2  \sum_{j \in A(x_i)} K(x_i(t) - x_j(t))  \nonumber\\
&\qquad \qquad   - \sum_{i=1}^N |\nabla V(x_i(t))| \sum_{j \notin A(x_i)}  |x_j - x_i| K(x_i - x_j) R(x_i,x_j),\nonumber
\end{align}
where 
\[
R(a,b) = \int_0^1 \norm{D^2 V(s a + (1 - s)b)}\,ds, \quad a,b \in \R^d.
\]
By our assumptions on $V$, we have 
\begin{align}
R(x_j,x_i) & \leq C + \max( \norm{D^2V(x_j)}, \norm{D^2V(x_i)}) \nonumber \\
& \leq C (1 + |x_j|^{p-2} + |x_i|^{p-2}) \nonumber \\
& \leq C (1 + |x_j - x_i|^{p-2} + |x_i|^{p-2}). \nonumber
\end{align}
(As elsewhere in the paper, the constant $C$ may change from line to line, here). By the assumptions on $V$, there is $\alpha = (p-2)/(p-1) \in [0,1)$ such that $|x_i|^{p-2} \leq C (1 + |\nabla V(x_i)|^\alpha$.  Also, $K(x_i -x_j)| |x_j - x_i|^{p-1}$ is bounded, by assumption.  Consequently,
\begin{align}
|K(x_i -x_j)| |x_j - x_i| R(x_j,x_i) &  \leq C (1 + |\nabla V(x_i)|^\alpha ), \quad \forall \;\; j \notin A(x_i). \label{KRassump}
\end{align}
Then
\begin{align}
S_2 & \geq  \sum_{i=1}^N \frac{1}{2} | \nabla V(x_i(t))|^2  \sum_{j \in A(x_i)} K(x_i(t) - x_j(t))  \nonumber \\
&\qquad \qquad -\sum_{i=1}^N |\nabla V(x_i(t))| (1 + |\nabla V(x_i(t))|^\alpha ) C N. \nonumber
\end{align}
Since $i \in A(x_i)$, the trivial bound 
\[
\sum_{j \in A(x_i)} K(x_i(t) - x_j(t)) \geq K(0) > 0
\]
always holds. Applying H\"older's inequality with exponents $p=2$ and with $(p,p^*) = (2/(1 + \alpha), 2/(1 - \alpha))$ we obtain
\begin{align}
S_2 & \geq   \frac{1}{2} \sum_{i=1}^N | \nabla V(x_i(t))|^2 K(0) - \epsilon \sum_i |\nabla V(x_i(t))|^2 - \frac{1}{\epsilon } C^2 N^3\nonumber  \\
& \quad \quad - \delta \sum_{i=1}^N| \nabla V(x_i(t))|^2 - \frac{1}{\delta^{(1+\alpha)/(1 - \alpha)}} N (CN)^{2/(1 - \alpha)}.\nonumber
\end{align}
In particular, there are constants $C_1, C_2 > 0$ (dependent on $K(0)$ and $\alpha$) so that
\begin{align}
S_2 & \geq C_1 \sum_{i=1}^N | \nabla V(x_i(t))|^2 - C_2 N^{3/(1 - \alpha)}. \nonumber
\end{align}

The sum $S_1$ is bounded by
\begin{align}
|S_1| & = \Big|\sum_{i,j=1}^N  \nabla K(x_i(t) - x_j(t))\cdot \nabla V(x_i(t))\Big| \nonumber \\
& \leq \frac{\epsilon}{2} \sum_{i=1}^N |\nabla V(x_i)|^2 + \frac{1}{2\epsilon} \sum_{i} |\sum_{j} \nabla K(x_i - x_j)|^2.\nonumber
\end{align}
Combining all these estimates, we obtain
\[
\partial_t H_N(x(t)) \leq - \frac{C_3}{N^2} \sum_{i=1}^N | \nabla V(x_i(t))|^2 + C_4 N^{3/(1 - \alpha) - 2}. 
\]
Since $|\nabla V|^2 \geq C( V + 1)- C'$ holds for all $x$, for some positive constants $C, C'$, this implies
\begin{align}
\partial_t H_N(x(t)) & \leq - \frac{C_5}{N^2} \sum_{i=1}^N (V(x_i(t) + 1)  + C_6 N^{3/(1 - \alpha) - 2}, \nonumber \\
& =  - \frac{C_5}{N} H_N(x(t))  + C_6 N^{3/(1 - \alpha) - 2},\nonumber
\end{align}
which implies that $H_N$ is uniformly bounded in $t$, for $N$ fixed.

\end{proof}

\section{Regularity for the mean-field PDE} \label{sec:wellpose}
In this section we prove Proposition \ref{prop:inteq} under the assumption that the initial
distribution $\nu \in \P_V$ has
a density $\rho_0 \in \sY^1_{k, V}$ with some fixed $k \geq 2$ and that 
the kernel $K$ is $(k+2)$ times differentiable with bounded derivatives.

\begin{proof}[Proof of Proposition \ref{prop:inteq}]
By Theorem \ref{thm:weak}, we know that $\rho_t \in C([0,T]; \mathscr{P}_V)$ satisfies
\begin{equation}
\| \rho_t \|_{\P_V} \leq e^{C_1 t} \| \nu \|_{\P_V}, \quad t \geq 0. \label{eq:vrho}
\end{equation}
Consequently, the vector field $(t,x) \mapsto U[\rho_t](x)$ (defined at \eqref{Uvecdef}) satisfies
\begin{align}
|U[\rho_t](x)| & \leq  \left( \| D K \|_\infty + \| \nabla K \|_\infty \right) e^{C_1 t} \| \rho_0 \|_{\P_V}, \nonumber \\
|D_x^j U[\rho_t](x)| & \leq  \left( \| D^{j+1} K \|_\infty + \| D^j K \|_\infty \right) e^{C_1 t} \| \rho_0 \|_{\P_V}, j= 1,2,\cdots, k+1.
\end{align}
Thus $U(t, x) \in C([0,T]; C_B^{k+1}(\R^d))$ where we recall that 
$C_B^{k+1}(\R^d)$ is the space of continuous functions with
bounded $(k+1)$-th order derivatives. 
Let $\Phi_t(x) = X(t,x,\nu)$ denote the characteristic flow (Definition \ref{def:mfcf}). Since $\Phi_t$ satisfies the ODE system $\frac{d}{dt} \Phi_t(x) = U[\rho_t](\Phi_t(x))$, it follows from standard theory that the maps $x \mapsto \Phi_t$ and its inverse $\Phi_t^{-1}$ are both $C^k$ maps (e.g. see Chapter 2 of \cite{Teschl12}). Therefore, if $\rho_0$ has a density, then $\rho_t$ also has a density. In fact, $\rho_t(x)$ is given by
\[
\rho_t(x) = (\Phi_t)_\# \rho_0 = \rho_0(\Phi^{-1}_t x) \exp \left(-\int_0^t (\nabla_x \cdot U[\rho_s])( \Phi_s \circ \Phi_t^{-1} x)\,ds \right)
\]
Moreover, since $\rho$ satisfies 
$$
\partial_t \rho_t = - \nabla \cdot(\rho_t U[\rho_t])
$$
with the vector field $U(t, x) \in C([0,T]; C_B^{k+1}(\R^d))$, it follows from \cite[Lemma 2.8]{laurent2007local}
that $\rho \in C([0,T]; H^k(\R^d))$ for any $T > 0$ and $\partial_t \rho \in C([0,T]; H^{k-1}(\R^d))$.

It remains to prove that $\rho(t, \cdot) \in W^{1,1}_V$ for every $t\in [0, T]$ 
and that it satisfies the a priori estimate \eqref{eq:apriori1}. First, we show that $\rho(t,\cdot) \in W^{1,1}_V$. To see this, we differentiate both sides of 
\begin{equation}
\partial_t \rho = - \nabla \cdot (\rho U[\rho_t]) \label{eq:iteration} 
\end{equation}
with respect to $x_i$ to get the following equation for $\partial_{x_i} \rho$ 
\begin{equation}
\partial_t (\partial_{x_i} \rho) = - \nabla  \cdot ((\partial_{x_i}\rho) U[\rho_t]) - \nabla  \cdot (\rho ((\partial_{x_i} U[\rho_t])) 
\end{equation}
Now given $\delta > 0$, we define the one dimensional function
\be\label{eq:phidelta}
\phi_\delta(x) = \sqrt{|x|^2 + \delta}.
\en
It is clear that $\phi_\delta(x) \gt |x|$ as $\delta \gt 0$ and that 
$\sup_{x\in \R}|\phi_\delta^\prime(x)| \leq 1$. Then $\phi_\delta(\partial_{x_i} \rho)$ satisfies:
\begin{equation}
\partial_t \phi_\delta(\partial_{x_i} \rho) = - \phi_\delta'(\partial_{x_i} \rho) \nabla  \cdot ((\partial_{x_i}\rho) U[\rho_t]) - \phi_\delta'(\partial_{x_i} \rho) \nabla  \cdot (\rho ((\partial_{x_i} U[\rho_t])) 
\end{equation}
Notice that since $\rho \in C([0,T]; H^k(\R^d))$ for any $T > 0$ and
$\partial_t \rho \in C([0,T]; H^{k-1}(\R^d))$, the above equation holds in the space 
$C([0,T]; L^2(\R^d))$. Let $\eta_R$ be a smooth cut-off function on $\R^d$ such that
\[
\eta_R(x) = \eta(x/R), \quad \text{ where } \eta \in C^\infty_c(\R^d) 
\text{ and } \eta(x) = 1 \text{ for } |x| \leq 1, \eta(x) = 0 \text{ for } 
|x| \geq 2.
\]
Next, we multiply the above equation with $(1 + V)\eta_R$, and then integrate on the whole space to get 
\begin{eqnarray}
\partial_t \int_{\R^d} (1 + V(x)) \eta_R(x) \phi_\delta(\partial_{x_i} \rho) \,dx & = & - \int_{\R^d} (1 + V(x)) \eta_R(x)  \nabla (\phi_\delta(\partial_{x_i} \rho)) \cdot U[\rho_t] \,dx \nonumber \\
&& - \int_{\R^d} (1 + V(x)) \eta_R(x)  \phi_\delta'(\partial_{x_i} \rho)  \partial_{x_i}\rho \nabla \cdot U[\rho_t] \,dx \nonumber \\
& & - \int_{\R^d} (1 + V(x)) \eta_R(x)  \phi_\delta'(\partial_{x_i} \rho) \nabla \rho \cdot (\partial_{x_i} U[\rho_t])  \,dx \nonumber \\
& & - \int_{\R^d} (1 + V(x)) \eta_R(x)  \phi_\delta'(\partial_{x_i} \rho) \rho \nabla \cdot (\partial_{x_i} U[\rho_t])  \,dx \nonumber \\
& =: & \sum_{j=1}^4 I_j(\delta, R). \label{eq:phideltav}
\end{eqnarray}
Using the fact that $|\phi^\prime_\delta| \leq 1$ and 
that $\eta_R$ is uniformly bounded, we have that
$$
I_j(\delta, R) \leq C(V)\|K\|_{3,\infty} \|\rho(t,\cdot)\|_{W^{1,1}_V}
\|\rho(t,\cdot)\|_{L^1_V} \text{ for } j=2,3,4.
$$
For $I_1(\delta, R)$, using integration by parts and the assumption \eqref{eq:asv3} one obtains that
$$
\bea
I_1(\delta, R) & = -\int_{\R^d}  \phi_\delta(\partial_{x_i} \rho) \nabla \cdot \Big((1 + V)\eta_R\  (\nabla K \ast \rho + \nabla K \ast  (\nabla V\rho))\Big)\\
& \leq C(V)\|K\|_{3,\infty} \|\rho(t,\cdot)\|_{L^1_V} \int_{\R^d} (1 + V) \phi_\delta(\partial_{x_i} \rho)  .
\ena
$$
Consequently, letting $\delta \gt 0$ and $R\gt\infty$, we obtain from \eqref{eq:phideltav} and \eqref{eq:vrho} that 
\begin{eqnarray}
\partial_t \|(1+V)\nabla \rho(t,\cdot)\|_{L^1} & \leq & C(V)\|K\|_{3,\infty}\|\rho(t,\cdot)\|_{L^1_V} \|(1+V)\nabla \rho(t,\cdot)\|_{L^1}. \nonumber \\
& \leq & C e^{C_1 t} \| \nu \|_{\P_V} \|(1+V)\nabla \rho(t,\cdot)\|_{L^1}. 
\end{eqnarray}
This implies
\be
\|\rho(t, \cdot)\|_{W^{1,1}_V} \leq \exp\Big(C_2 (e^{C_1 t} - 1) \| \nu \|_{\P_V} \Big) \|\rho_0\|_{W^{1,1}_V}.
\en

%%%%%%%%%%%%%%%%%%%%%%%%%%%%%%%%%%%%%
\commentout{

$$
\bea
 \partial_t (\partial_{x_i} \rho) & = \nabla (\partial_{x_i} \rho) \cdot \Big(\nabla K \ast \rho + K \ast (\nabla V \rho)\Big)
  + \nabla \rho \cdot \Big(\nabla \partial_{x_i} K \ast \rho + \partial_{x_i} K \ast (\nabla V \rho)\Big)\\
  & + \partial_{x_i} \rho \cdot \Big(\Delta K \ast \rho+ \nabla K \ast (\nabla V \rho)\Big) + \rho \Big(\Delta \partial_{x_i} K \ast \rho + \nabla \partial_{x_i} K \ast (\nabla V \rho)\Big).
\ena
$$

Notice that since $\rho\in \sX_k(0,T)$ with $k\geq 2$ the above equation holds in the space $L^\infty((0,T); L^2(\R^d))$.
Now given $\delta > 0$, we define the one dimensional function
\be
\phi_\delta(x) = \sqrt{|x|^2 + \delta}.
\en
It is clear that $\phi_\delta(x) \gt |x|$ as $\delta \gt 0$ and that 
$\sup_{x\in \R}|\phi_\delta^\prime(x)| \leq 1$. Then $\phi_\delta(\partial_{x_i} \rho)$ satisfies the following equation in the weak sense:
$$
\bea
&  \partial_t (\phi_\delta(\partial_{x_i} \rho))  = \phi_\delta^\prime(\partial_{x_i} \rho)\nabla (\partial_{x_i} \rho) \cdot \Big(\nabla K \ast \rho + K \ast (\nabla V \rho)\Big)\\
 & + \phi_\delta^\prime(\partial_{x_i} \rho) \nabla \rho \cdot \Big(\nabla \partial_{x_i} K \ast \rho + \partial_{x_i} K \ast (\nabla V \rho)\Big)
   + \phi_\delta^\prime(\partial_{x_i} \rho) \partial_{x_i} \rho \cdot \Big(\Delta K \ast \rho + \nabla K \ast (\nabla V \rho)\Big) \\
   & + \phi_\delta^\prime(\partial_{x_i} \rho) \rho \Big(\Delta \partial_{x_i} K \ast \rho+ \nabla \partial_{x_i} K \ast (\nabla V \rho)\Big).
\ena
$$
Let $\eta_R$ be a smooth cut-off function on $\R^d$ such that
\[
\eta_R(x) = \eta(x/R), \quad \text{ where } \eta \in C^\infty_c(\R^d) 
\text{ and } \eta(x) = 1 \text{ for } |x| \leq 1, \eta(x) = 0 \text{ for } 
|x| \geq 2.
\]
Next, we multiply the above equation
with $(1 + V)\eta_R$, and then integrate on the whole space to get 
\be
\bea
& \partial_t \int_{\R^d} (1 + V)\eta_R\ \phi_\delta(\partial_{x_i} \rho) dx = 
\int_{\R^d} (1 + V)\eta_R\ \nabla (\phi_\delta(\partial_{x_i} \rho))  \cdot \Big(\nabla K \ast \rho + K \ast (\nabla V \rho)\Big) dx\\
& \qquad + \int_{\R^d} (1 + V)\eta_R\ \phi_\delta^\prime(\partial_{x_i} \rho) \nabla \rho \cdot \Big(\nabla \partial_{x_i} K \ast \rho + \partial_{x_i} K \ast (\nabla V \rho)\Big) dx\\
& \qquad +  \int_{\R^d} (1 + V)\eta_R\ \phi_\delta^\prime(\partial_{x_i} \rho) \partial_{x_i} \rho \cdot \Big(\Delta K \ast \rho + \nabla K \ast (\nabla V \rho)\Big)  dx\\
& \quad +  \int_{\R^d} (1 + V)\eta_R\ \phi_\delta^\prime(\partial_{x_i} \rho) \rho \Big(\Delta \partial_{x_i} K \ast \rho + \nabla \partial_{x_i} K \ast (\nabla V \rho)\Big)dx =: \sum_{j=1}^4 I_j(\delta, R).
\ena
\en
}%end commentout
%%%%%%%%%%%%%%%%%%%%%%%%%%%%%%%%%%%%%%%
%%%%%%%%%%%%%%%%%%%%%%%%%%%%%%%%%%%%%%%

Finally we derive an $H^k$-estimate for the solution. For doing so, let $\balpha$ be a multi-index such that $|\balpha| \leq k$. Taking $\partial^{\balpha}$ on the both sides of \eqref{eq:iteration}, multiplying the resulting equation with $\partial^{\balpha} \rho$ and then integrating gives
\be\label{eq:hk}
\bea
& \frac{1}{2} \partial_t \|\partial^{\balpha} \rho\|^2_{L^2(\R^d)} = \int_{\R^d} \partial^{\balpha} \Big(\nabla \rho \cdot \nabla K \ast \rho\Big)  \partial^{\balpha} \rho dx+ \int_{\R^d} \partial^{\balpha} \Big(\rho \Delta K \ast \rho \Big)  \partial^{\balpha} \rho dx\\
&  +  \int_{\R^d} \partial^{\balpha} \Big( \nabla \rho \cdot K \ast (\nabla V\rho) \Big) \partial^{\balpha} \rho dx + \int_{\R^d}  \partial^{\balpha}\Big(\rho  \nabla K\ast (\nabla V\rho)\Big) \partial^{\balpha} \rho dx
 =: \sum_{i=1}^4 J_i.
\ena
\en
Note that by Leibniz rule and integration by parts, 
$$
\bea
& J_1 = \int_{\R^d} \sum_{\bbeta < \balpha} \partial^{\bbeta} (\nabla \rho) \cdot  (\nabla ( \partial^{\balpha - \bbeta} K) \ast \rho) \partial^\alpha \rho dx - \frac{1}{2} \int_{\R^d} |\partial^\alpha \rho|^2 \Delta K \ast \rho dx\\
& \leq \sum_{\bbeta < \balpha } \|\partial^{\bbeta} (\nabla \rho)\|_{L^2} 
\|\nabla ( \partial^{\balpha - \bbeta} K) \ast \rho\|_{L^\infty} 
\|\partial^{\balpha} \rho\|_{L^2} + \frac{1}{2} \|\partial^{\balpha} \rho\|_{L^2}^2 \| \Delta K \ast \rho\|_{L^1}\\
& \leq C(K) \|\rho\|_{H^k(\R^d)}^2 \|\rho\|_{L^1}.
\ena
$$
Similarly, we have for $i=2,3,4$, 
$$
J_i \leq C(V, K) \|\rho\|_{H^k(\R^d)}^2 \|\rho\|_{L^1_V}.
$$
Plugging the estimates into \eqref{eq:hk} and using \eqref{eq:vrho}, we obtain by summing over $\balpha$ with $|\balpha| \leq k$ that
$$
\partial_t \|\rho(t,\cdot)\|_{H^k(\R^d)} \leq Ce^{C_1 t} \|\nu\|_{\mathscr{P}_V}\|\rho(t,\cdot)\|_{H^k(\R^d)}
$$
which implies 
\be\label{eq:hk2}
\|\rho(t, \cdot)\|_{H^k(\R^d)} \leq \exp\Big(C_2 (e^{C_1 t} - 1) \| \nu \|_{\P_V} \Big) \|\rho_0\|_{H^k(\R^d)}.
\en
The estimate follows from \eqref{eq:apriori1} and \eqref{eq:hk2}. This finishes the proof of the proposition.

\end{proof}

\section{Stability estimate}\label{sec:cps}

In this section we prove \cref{thm:stability} using Dobrushin's coupling argument, following Theorem 1.4.1 of \cite{golse2016dynamics}.

\begin{proof}[Proof of \cref{thm:stability}] Recall that $p = q^\ast = \frac{q}{q-1} $. First by the assumption that $\|\nu_i\|_{\P_p} \leq R < \infty$ and the fact that $\P_p\subset\P_V$ thanks to \cref{vpgrowth}, we know that  there exists $C(R) > 0$ such that
\be\label{eq:nuiPv}
\|\nu_i\|_{\P_V} \leq C(R) < \infty.
\en
By the proof of Theorem \ref{thm:weak} and Definition \ref{def:mfcf} of the mean field characteristic flow, we know that the weak solutions $\mu_{i,t}$ take the form
\[
\mu_{i,t} = (X(t, \cdot, \nu_i))_{\#} \nu_i, \quad i = 1,2
\]
So, we must estimate $ \W_p^p(\mu_{1,t}, \mu_{2,t})$ in terms of $ \W_p^p(\nu_1, \nu_2)$. Let $\pi^0$ be a coupling measure between the probability measures $\nu_1 $ and $ \nu_2$. Define for $\delta > 0$, $\phi_\delta(x) = \frac{1}{p}(|x|^2 + \delta)^{p/2}$ to be an approximation to $\frac{1}{p}|x|^p$, 
 Given any two points $x_1, x_2\in \R^d$, we have from \eqref{eq:mfcf} that
 $$
 \bea
 & \partial_t \phi_\delta\Big(X(t, x_1, \nu_1) - X(t, x_2, \nu_2)\Big) = - \nabla \phi_\delta \Big(X(t, x_1, \nu_1) - X(t, x_2, \nu_2)\Big)\times \\
 &  \times \bigg\{\Big( \int_{\R^d} \nabla K(X(t, x_1, \nu_1) - X(t, x_1^\prime, \nu_1)) \nu_1(dx_1^\prime)\\
 &  \qquad\qquad -  \int_{\R^d} \nabla K(X(t, x_2, \nu_2) - X(t, x_2^\prime, \nu_2)) \nu_2(dx_2^\prime)\Big)\\
 & \quad - \Big(\int_{\R^d}  K(X(t, x_1, \nu_1) - X(t, x_1^\prime, \nu_1)) \nabla V(X(t, x_1^\prime, \nu_1))\nu_1(dx_1^\prime)\\
 & \qquad\qquad - \int_{\R^d}  K(X(t, x_2, \nu_2) - X(t, x_2^\prime, \nu_2)) \nabla V(X(t, x_2^\prime, \nu_2))\nu_2(dx_2^\prime) \Big) \bigg\}\\
 &= -\nabla \phi_\delta \Big(X(t, x_1, \nu_1) - X(t, x_2, \nu_2)\Big)\times \\
 &\times \bigg\{   \int_{\R^{2d}} \Big(\nabla K(X(t, x_1, \nu_1) - X(t, x_1^\prime, \nu_1)) \\
 &  \qquad\qquad
 - \nabla K(X(t, x_2, \nu_2) - X(t, x_2^\prime, \nu_2))\Big) \pi^0(dx_1^\prime dx_2^\prime)\\
 & \quad +  \int_{\R^{2d}} \Big(K(X(t, x_1, \nu_1) - X(t, x_1^\prime, \nu_1))
 -  K(X(t, x_2, \nu_2) - X(t, x_2^\prime, \nu_2))\Big)\\
 & \qquad \qquad \times \nabla V(X(t, x_1^\prime, \nu_1))  \pi^0(dx_1^\prime dx_2^\prime)\\
 & \quad +\int_{\R^{2d}}  K(X(t, x_2, \nu_2) - X(t, x_2^\prime, \nu_2))\\
 &   \qquad\qquad   \times\Big( \nabla V(X(t, x_1^\prime, \nu_1)) - \nabla V(X(t, x_2^\prime, \nu_2)) \Big)\pi^0(dx_1^\prime dx_2^\prime) 
 \Big)\bigg\}\\
 & =:  I_1 + I_2 + I_3.
 \ena
 $$
 Below we bound $I_i$ individually. First, it is important to notice that 
 $$
| \nabla \phi_\delta (x)|  = |(|x|^2 + \delta)^{p/2 - 1} x| \leq  |x|^{p-1}.
 $$
  Then thanks to Assumption \eqref{ass:k} on $K$ and the fact that
  the inclusion $L^p \xhookrightarrow{} L^1$ is bounded for $p>1$, we have  
 $$
 \bea
 &I_1 \leq \|K\|_{2,\infty} \Big|X(t, x_1, \nu_1) - X(t, x_2, \nu_2)\Big|^{p}  \\
 & \qquad + \|K\|_{2,\infty}\Big|X(t, x_1, \nu_1) - X(t, x_2, \nu_2)\Big|^{p-1}\\
 &\qquad \qquad \times 
 \Big( \int_{\R^{2d}} \Big|X(s, x_1^\prime, \nu_1) - X(s, x_2^\prime, \nu_2)\Big|^p \pi(dx_1^\prime d x_2^\prime)\Big)^{1/p} .
 \ena
 $$
 For $I_2$, it follows from Assumption \ref{ass:v} (A2)-(A3) and H\"older's inequality that 
 $$
 \bea
  & I_2 \leq \|K\|_{1,\infty} \Big|X(t, x_1, \nu_1) - X(t, x_2, \nu_2)\Big|^{p}\times \int_{\R^{d}} \Big| \nabla V(X(t, x_1^\prime, \nu_1)) \Big|\nu_1(dx_1^\prime)  \\
 & +   \|K\|_{1,\infty}\Big|X(t, x_1, \nu_1) - X(t, x_2, \nu_2)\Big|^{p-1}\times  \int_{\R^{2d}}\Big|X(t, x_1^\prime, \nu_1) - X(t, x_2^\prime, \nu_2)\Big| \\
 & \qquad\qquad \times\Big| \nabla V(X(t, x_1^\prime, \nu_1)) \Big|\pi^0(dx_1^\prime d x_2^\prime)\\
& \leq  C_V \|K\|_{1,\infty} \Big|X(t, x_1, \nu_1) - X(t, x_2, \nu_2)\Big|^{p} \int_{\R^d} (1 + V(X(t, x_1^\prime, \nu_1))) \nu_1(dx_1^\prime) \\
&  + \|K\|_{1,\infty} \Big|X(t, x_1, \nu_1) - X(t, x_2, \nu_2)\Big|^{p-1}\\
& \qquad \qquad \times \Big( \int_{\R^{2d}} \Big|X(t, x_1^\prime, \nu_1) - X(t, x_2^\prime, \nu_2)\Big|^p \pi^0(dx_1^\prime d x_2^\prime)\Big)^{1/p}\\
&   \qquad \qquad \times 
\Big( \int_{\R^{2d}} \Big|  \nabla V(X(t, x_1^\prime, \nu_1))  \Big|^{q}
\mu^0(dx_1^\prime )\Big)^{1/q}. 
  \ena
  $$
  Observe that the integrals involving $V$ on the right side of above can be bounded in exactly the same way as \eqref{eq:vxbd}. Hence we can obtain
  $$
  \bea
  & I_2  \leq C_V e^{Ct} \|K\|_{1,\infty}  \|\mu^0\|_{\P_V} \Big(\Big|X(t, x_1, \nu_1) - X(t, x_2, \nu_2)\Big|^{p} \\
  & + \Big|X(t, x_1, \nu_1) - X(t, x_2, \nu_2)\Big|^{p-1} \cdot \Big( \int_{\R^{2d}} \Big|X(t, x_1^\prime, \nu_1) - X(t, x_2^\prime, \nu_2)\Big|^p \pi^0(dx_1^\prime d x_2^\prime)\Big)^{1/p} \Big)
  \ena
  $$
  with the constant $C$ depending only on $V$.
  Finally, we find an upper bound for $I_3$. In fact, an application of the intermediate value theorem to the difference of $\nabla V$ and the inequality \eqref{eq:v2} of Assumption \ref{ass:v} (A-2) yields that
  $$
  \bea
  & I_3 \leq \|K\|_{\infty} \Big|X(t, x_1, \nu_1) - X(t, x_2, \nu_2)\Big|^{p-1} \int_{\R^{2d}}  \Big|X(t, x_1^\prime, \nu_1) - X(t, x_2^\prime, \nu_2)\Big| \\
  & \qquad \times \sup_{\theta\in [0, 1]} \Big|\nabla V^2(\theta X(t, x_1^\prime, \nu_1)  + (1-\theta) X(t, x_2^\prime, \nu_2))\Big| \pi^0(dx_1^\prime dx_2^\prime)\\
  & \leq C_V\|K\|_{\infty}\Big|X(t, x_1, \nu_1) - X(t, x_2, \nu_2)\Big|^{p-1}\\
  &  \qquad \qquad \times \Big(  \int_{\R^{2d}}  \Big|X(t, x_1^\prime, \nu_1) - X(t, x_2^\prime, \nu_2)\Big|^{p}
  \pi^0(dx_1^\prime dx_2^\prime)\Big)^{1/p}  \\
  &  \qquad\qquad \times \Big(   \int_{\R^{2d}} \Big(1 + V(X(t, x_1^\prime,
  \nu_1)) + V(X(t, x_2^\prime, \nu_2))\Big) \pi^0(dx_1^\prime dx_2^\prime) 
  \Big)^{1/q}\\
  & \leq e^{Ct} C_V\|K\|_{\infty} (\|\mu^0_1\|_{\P_V} +\|\mu^0_2\|_{\P_V} )\Big|X(t, x_1, \nu_1) - X(t, x_2, \nu_2)\Big|^{p-1}\\
  &\qquad \qquad \times \Big(  \int_{\R^{2d}} \Big|X(t, x_1^\prime, \nu_1) - X(t, x_2^\prime, \nu_2)\Big|^{p}
  \pi^0(dx_1^\prime dx_2^\prime)\Big)^{1/p}.
  \ena
  $$
  If we define 
  $$
  D_p(\pi)(s) := \Big( \int_{\R^{2d}} \Big|X(s, x_1^\prime, \nu_1) - X(s, x_2^\prime, \nu_2)\Big|^p \pi(dx_1^\prime d x_2^\prime)\Big)^{1/p},
  $$
 then by combing the estimates above, we obtain that for any $t\in [0, T]$,
  $$
\bea
& \phi_\delta\Big(X(t, x_1, \nu_1) - X(t, x_2, \nu_2)\Big) =
\phi_\delta(x_1 - x_2) + 
\int_{0}^t \partial_s \phi_\delta\Big(X(s, x_1, \nu_1) - X(s, x_2, \nu_2)\Big) ds\\
&  \leq \phi_\delta(x_1 - x_2)  +  C(K, V)e^{CT}(\|\nu_1\|_{\P_V} + \|\nu_2\|_{\P_V}) \int_{0}^t \Big(\Big|X(s, x_1, \nu_1) - X(s, x_2, \nu_2)\Big|^{p} 
\\
& +  \Big|X(s, x_1, \nu_1) - X(s, x_2, \nu_2)\Big|^{p-1} \cdot D_p(\pi^0)(s) \Big)ds.
 \ena
  $$
  Now integrating the above inequality with respect to the coupling $\pi^0(dx_1 dx_2)$, using the fact that
  $$
    \int_{\R^{2d}}  \Big|X(s, x_1, \nu_1) - X(s, x_2, \nu_2)\Big|^{p-1}\pi^0(dx_1 dx_2) \leq D^{p-1}_p(\pi^0)(s)
  $$ and finally letting $\delta \gt 0$ yields
  $$
  D^p_p(\pi^0)(t) \leq D^p_p(\pi^0)(0) + C(K, V)e^{CT}(\|\nu_1\|_{\P_V} + \|\nu_2\|_{\P_V}) \int_{0}^t D^p_p(\pi^0)(s)ds.
  $$
  % \int_{\R^{2d}} |x_1^\prime - x_2^\prime|^p \pi^0(dx_1^\prime dx_2^\prime) 
  By the Gr\"onwall's inequality we obtain that 
  $$
  D_p^p(\pi^0)(t) \leq D^p_p(\pi^0)(0) \exp\Big( C(K, V)e^{CT}(\|\nu_1\|_{\P_V} + \|\nu_2\|_{\P_V}) t\Big). 
  $$
  Now since $\pi^0\in \Gamma(\nu_1, \nu_2)$ and $\mu_{i,t}  = (X(t, \cdot, \nu_i))_{\#} \nu_i$,  the mapping 
  $$
  \Xi_t  : (x_1, x_2)\in \R^{2d} \mapsto (X(t, x_1, \nu_1), X(t, x_2, \nu_2)) \in \R^{2d} 
  $$
  satisfies that $(\Xi_t)_{\#} \pi^0 \in \Gamma(\mu_{1,t}, \mu_{2,t})$.  As a consequence, we have that 
  $$
  \bea
   \W_p^p(\mu_{1,t}, \mu_{2,t}) & = \inf_{\pi\in  \Gamma(\mu_{1,t}, \mu_{2,t})} \int_{\R^{2d}}|x_1 - x_2|^p  \pi(dx_1 dx_2)\\
  &  \leq \inf_{\pi^0\in  \Gamma(\nu_1, \nu_2)} D_p^p(\pi^0)(t) \\
    &  \leq \text{exp}\Big(C(K, V) e^{CT}(\|\nu_1\|_{\P_V} + \|\nu_2\|_{\P_V})
    t\Big)\cdot \inf_{\pi^0\in  \Gamma(\nu_1, \nu_2)}D^p_p(\pi^0)(0)
\\
& = \text{exp}\Big(C(K, V) e^{CT}(\|\nu_1\|_{\P_V} + \|\nu_2\|_{\P_V}) t\Big)\cdot \W_p^p(\nu_1, \nu_2).
\ena 
  $$
  This finishes the proof in view of \cref{eq:nuiPv}.
\end{proof}

\section{Long time behavior of the solution of the mean field PDE} \label{sec:asymp}

In this section we prove \cref{thm:asymp}. For doing so, we recall following extra assumption on the kernel $K$: 
\[
K = K_{1/2}\ast K_{1/2}  \text{ with } K_{1/2}\in \mathcal{S} \text{ and } \hat{K}_{1/2}(\xi) \neq 0,\ \forall \xi \in \R^d.
\]
A canonical kernel satisfying this condition is a Gaussian kernel. 
%\[
%K(x) = K_\sigma(x) := (4\pi \sigma)^{d/2} e^{-\frac{|x|^2}{4\sigma}}.
%\]
%For the sake of notational simplicity, from now on we assume that $K(x) = K_1(x)$.

\begin{proof}[Proof of  \cref{thm:asymp}]
To prove $\rho_t \wgt \rho_\infty$ as $t\gt \infty$, we only need to 
prove that $\rho_{t_k} \wgt \rho_\infty$ for any sequence $t_k\nearrow \infty$. 
Indeed, suppose that the later is true and that $\rho_t$ does not converge weakly to $\rho_\infty$.
Then there exists a constant $\eps > 0$ and a bounded continuous function $\varphi$,
such that there exists a sequence $t_k\nearrow \infty$ such that 
$$
\Big|\int_{\R^d} \rho_{t_k} \varphi dx  - \int_{\R^d} \rho_{\infty} \varphi dx  \Big| \geq \eps,
$$
which contradicts with the assumption. 
To prove $\rho_{t_k} \wgt \rho_\infty$ for any sequence $t_k\nearrow \infty$, according to \cite[Theorem 2.6]{billingsley2013convergence}, it suffices to show that each subsequence of $\{\rho_{t_k}\}_{k\in\N}$, still denoted by $\{\rho_{t_{k}}\}_{k\in\N}$, has 
a further subsequence $\{\rho_{t_{k_m}}\}_{m\in\N}$  
converging weakly to $\rho_\infty$. Below we divide our proof into three steps.

\noindent {\bf Step 1}: Tightness of $\{\rho_{t_k}\}_{k\in\N}$. In fact, since $\rho_t$ solves \eqref{eq:mfl}, it is straightforward to check that 
 \be\label{eq:diffkl}
 \bea
\partial_t \KL (\rho_t\, ||\, \rho_\infty)&   = -\int_{\R^d} \int_{\R^d}
\rho_t(x) \rho_t(y) \nabla \log\Big(\frac{\rho_t}{\rho_\infty}(x)\Big)
\cdot K(x - y) \cdot \nabla \log\Big(\frac{\rho_t}{\rho_\infty}(y)\Big)dx dy\\
& =   - \int_{\R^d} \int_{\R^d}
(\nabla \rho_t + \nabla V\rho_t)(x)  \cdot K(x - y) \cdot  (\nabla \rho_t + \nabla V\rho_t)(y) dx dy \\
&\leq 0,
\ena
 \en
 where the inequality follows from the fact that $K(x-y)$ is positive definite. 
 Furthermore, noticing that 
 $$ 0 \leq 
-\int_0^t \partial_s 
  \KL(\rho_s\, ||\, \rho_\infty) ds =  \KL(\rho_0 \,||\, \rho_\infty)- \KL(\rho_t \,||\, \rho_\infty) < \infty,
 $$
 one can obtain that $\partial_t
  \KL(\rho_t \,||\, \rho_\infty) \gt 0$ as $t\gt \infty$. 
  As a result of \eqref{eq:diffkl}, we have
  \be\label{eq:diffkl2}
  \int_{\R^d} \int_{\R^d}
(\nabla \rho_t + \nabla V\rho_t)(x)  \cdot K(x - y) \cdot  (\nabla \rho_t + \nabla V\rho_t)(y) dx dy \gt 0 \text{ as } t \gt \infty.
  \en
  Since the relative entropy functional $\rho \mapsto \KL(\rho \,||\, \rho_\infty)$ has compact sub-level sets in the weak topology (see e.g. \cite[Lemma 1.4.3]{dupuis2011weak}), it follows from $\KL(\rho_{t_k} \,||\, \rho_\infty) \leq   \KL(\rho_0 \,||\, \rho_\infty) < \infty$ that $\{\rho_{t_k}\}_{k\in\N}$ is tight. 
 Consequently there exists a subsequence $t_{k_m} \uparrow \infty$ and $\bar{\rho}\in \P(\R^d)$ such that $\KL(\bar{\rho} \,||\, \rho_\infty)< \infty$ and $\rho_{t_{k_m}} \wgt \bar{\rho}$.  
 
\noindent {\bf Step 2:} We show that $\bar{\rho}$ satisfies  $$
  K_{1/2} \ast (\nabla \bar{\rho} + \nabla V \bar{\rho}) = 0
 $$  in the sense of distribution.
  To this end, using Fourier transform and the fact that $\widehat{K} = \widehat{K}_{1/2}^2 $ we can write 
$$
\bea
& \int_{\R^d} \int_{\R^d}
(\nabla \rho_{t_{k_m}} + \nabla V\rho_{t_{k_m}})(x)  \cdot K(x - y) \cdot  (\nabla \rho_{t_{k_m}} + \nabla V\rho_{t_{k_m}})(y) dx dy\\
& =  \int_{\R^d} \widehat{K}(\xi) \Big|\widehat{(\nabla \rho_{t_{k_m}} + \nabla V\rho_{t_{k_m}})}(\xi) \Big|^2 d\xi\\
& = \int_{\R^d}  \Big| \widehat{K}_{1/2}(\xi)\widehat{(\nabla \rho_{t_{k_m}} + \nabla V\rho_{t_{k_m}})}(\xi) \Big|^2 d\xi\\
& = \|K_{1/2} \ast (\nabla \rho_{t_{k_m}} + \nabla V\rho_{t_{k_m}})\|_{L^2(\R^d)}^2.
\ena
$$ 
 Note that we are allowed to take the Fourier transform because $\rho_{t_{k_m}}\in \sY^1_{2, V}$
 by \cref{thm:asymp} and the assumption that $\rho_0 \in \mathscr{Y}^1_{2,V}$. This together with \eqref{eq:diffkl2} implies that 
 $K_{1/2} \ast (\nabla \rho_{t_{k_m}} + \nabla V\rho_{t_{k_m}}) \gt 0$ in $L^2(\R^d)$. 
 On the other hand, using $\rho_{t_{k_m}} \wgt \bar{\rho}$
 with $ \bar{\rho} \in \P(\R^d)$ and integration by parts, one sees that
 $$
 \bea
 K_{1/2} \ast (\nabla \rho_{t_{k_m}} + \nabla V\rho_{t_{k_m}}) & = 
 \int_{\R^d} K_{1/2}(x - y)(\nabla \rho_{t_{k_m}} + \nabla V\rho_{t_{k_m}})(y)dy\\
& = \int_{\R^d} \nabla K_{1/2}(x- y) \rho_{t_{k_m}}(y) + K_{1/2}(x- y)\nabla V(y) \rho_{t_{k_m}}(y)dy\\
& \gt  \int_{\R^d} \nabla K_{1/2}(x- y) \bar{\rho} (y) + K_{1/2}(x- y)\nabla V(y) \bar{\rho} (y)dy.
\ena
 $$
 Therefore we have that $\int_{\R^d} \nabla K_{1/2}(x- y) \bar{\rho} (y) + K_{1/2}(x- y)\nabla V(y) \bar{\rho} (y)dy = 0$ a.e. $x\in \R^d$. This in particular, implies  that  
 \be\label{eq:distribution}
  K_{1/2} \ast (\nabla \bar{\rho} + \nabla V \bar{\rho}) = 0
 \en
 in the sense of distribution. 
 
\noindent {\bf Step 3}: We show that $\bar{\rho} = \rho_\infty$. We first prove that $\nabla \bar{\rho} + \nabla V \bar{\rho} = 0 $ in the sense of tempered 
 distribution. In fact, since $\bar{\rho}\in \mathscr{P}(\R^d)$  and since $V$ grows at most polynomially (due to  \cref{ass:v} (A2)), we know that $(\nabla \bar{\rho} + \nabla V \bar{\rho}) \in \mathcal{S}^\prime$. Since $K_{1/2}\in \mathcal{S}$, it follows from the convolution theorem of Fourier transform (see e.g. \cite[Chapter 4.11, Theorem 3 and Proposition 7]{H66}) that 
 $
   K_{1/2} \ast (\nabla \bar{\rho} + \nabla V \bar{\rho})
 $ can be understood as a rapidly decreasing distribution whose Fourier transform is given by 
 $$
    \widehat{K_{1/2} \ast (\nabla \bar{\rho} + \nabla V \bar{\rho})} = \hat{K}_{1/2} \cdot \widehat{(\nabla \bar{\rho} + \nabla V \bar{\rho})}. 
 $$
By the assumption that $\hat{K}_{1/2} \neq 0$, we have from \eqref{eq:distribution} that $\widehat{\nabla \bar{\rho} + \nabla V \bar{\rho}} = 0$ and hence $\nabla \bar{\rho} + \nabla V \bar{\rho} = 0$. This in addition implies that 
$\nabla (e^{V} \bar{\rho}) = 0$ in the sense of distribution. Therefore $\bar{\rho} = C\rho_\infty$ a.e. for some constant $C$. Finally since both $\bar{\rho}$ and $\rho_\infty$ are probability density, $C =1$ and $\bar{\rho} = \rho_\infty$ a.e.  This finishes the proof. 
\end{proof}

\section*{Acknowledgement}
 The authors would like to thank the anonymous referees for their valuable comments and suggestions to improve the structure and quality of the paper.

\bibliographystyle{siamplain}
\bibliography{steinref}

\begin{thebibliography}{10}

\bibitem{billingsley2013convergence}
{\sc P.~Billingsley}, {\em Convergence of probability measures}, John Wiley \&
  Sons, 2nd~ed., 2013.

\bibitem{bobkov2014one}
{\sc S.~Bobkov and M.~Ledoux}, {\em One-dimensional empirical measures, order
  statistics and {K}antorovich transport distances}, Mem. Amer. Math. Soc.,
  (to appear).

\bibitem{bou2012nonasymptotic}
{\sc N.~Bou-Rabee and M.~Hairer}, {\em Nonasymptotic mixing of the mala
  algorithm}, IMA Journal of Numerical Analysis, 33 (2012), pp.~80--110.

\bibitem{burger2008large}
{\sc M.~Burger and M.~Di~Francesco}, {\em Large time behavior of nonlocal
  aggregation models with nonlinear diffusion}, Networks \& Heterogeneous
  Media, 3 (2008), pp.~749--785.

\bibitem{burger2013stationary}
{\sc M.~Burger, M.~d. Francesco, and M.~Franek}, {\em Stationary states of
  quadratic diffusion equations with long-range attraction}, Communications in
  Mathematical Sciences, 11 (2013), pp.~709--738.

\bibitem{CarilloCraigPatacchini17}
{\sc J.~A. Carillo, K.~Craig, and S.~Patacchini~Francesco}, {\em A blob method
  for diffusion}, preprint, arXiv: 1709.09195,  (2017).

\bibitem{carrillo2003kinetic}
{\sc J.~A. Carrillo, R.~J. McCann, C.~Villani, et~al.}, {\em Kinetic
  equilibration rates for granular media and related equations: entropy
  dissipation and mass transportation estimates}, Revista Matematica
  Iberoamericana, 19 (2003), pp.~971--1018.

\bibitem{chertock2017practical}
{\sc A.~Chertock}, {\em A practical guide to deterministic particle methods},
  in Handbook of Numerical Analysis, vol.~18, Elsevier, 2017, pp.~177--202.

\bibitem{chertock2001particle}
{\sc A.~Chertock and D.~Levy}, {\em Particle methods for dispersive equations},
  Journal of Computational Physics, 171 (2001), pp.~708--730.

\bibitem{craig2016blob}
{\sc K.~Craig and A.~Bertozzi}, {\em A blob method for the aggregation
  equation}, Mathematics of Computation, 85 (2016), pp.~1681--1717.

\bibitem{dalalyan2017theoretical}
{\sc A.~S. Dalalyan}, {\em Theoretical guarantees for approximate sampling from
  smooth and log-concave densities}, Journal of the Royal Statistical Society:
  Series B (Statistical Methodology), 79 (2017), pp.~651--676.

\bibitem{degond1989weighted}
{\sc P.~Degond and S.~Mas-Gallic}, {\em The weighted particle method for
  convection-diffusion equations. i. the case of an isotropic viscosity},
  Mathematics of Computation, 53 (1989), pp.~485--507.

\bibitem{degond1990deterministic}
{\sc P.~Degond and F.-J. Mustieles}, {\em A deterministic approximation of
  diffusion equations using particles}, SIAM Journal on Scientific and
  Statistical Computing, 11 (1990), pp.~293--310.

\bibitem{dobrushin1979vlasov}
{\sc R.~L. Dobrushin}, {\em Vlasov equations}, Functional Analysis and Its
  Applications, 13 (1979), pp.~115--123.

\bibitem{dupuis2011weak}
{\sc P.~Dupuis and R.~S. Ellis}, {\em A weak convergence approach to the theory
  of large deviations}, vol.~902, John Wiley \& Sons, 2011.

\bibitem{durmus2017nonasymptotic}
{\sc A.~Durmus and E.~Moulines}, {\em Nonasymptotic convergence analysis for
  the unadjusted {L}angevin algorithm}, The Annals of Applied Probability, 27
  (2017), pp.~1551--1587.

\bibitem{fournier2015rate}
{\sc N.~Fournier and A.~Guillin}, {\em On the rate of convergence in
  {W}asserstein distance of the empirical measure}, Probability Theory and
  Related Fields, 162 (2015), pp.~707--738.

\bibitem{golse2016dynamics}
{\sc F.~Golse}, {\em On the dynamics of large particle systems in the mean
  field limit}, in Macroscopic and Large Scale Phenomena: Coarse Graining, Mean
  Field Limits and Ergodicity, Springer, 2016, pp.~1--144.

\bibitem{goodman1990convergence}
{\sc J.~Goodman, T.~Y. Hou, and J.~Lowengrub}, {\em Convergence of the point
  vortex method for the 2-{D} {E}uler equations}, Communications on Pure and
  Applied Mathematics, 43 (1990), pp.~415--430.

\bibitem{horstmann20031070}
{\sc D.~Horstmann}, {\em From 1970 until present: the {K}eller-{S}egel model in
  chemotaxis and its consequences}, Jahresber. Dtsch. Math.-Ver., 105 (2003),
  pp.~103--165.

\bibitem{H66}
{\sc J.~Horv\'{a}th}, {\em Topological vector spaces and distributions. {V}ol.
  {I}}, Addison-Wesley Publishing Co., Reading, Mass.-London-Don Mills, Ont.,
  1966.

\bibitem{JKO}
{\sc R.~Jordan, D.~Kinderlehrer, and F.~Otto}, {\em {The Variational
  Formulation of the Fokker--Planck Equation}}, SIAM Journal on Mathematical
  Analysis, 29 (1998), pp.~1--17.

\bibitem{keller1970initiation}
{\sc E.~F. Keller and L.~A. Segel}, {\em Initiation of slime mold aggregation
  viewed as an instability}, Journal of theoretical biology, 26 (1970),
  pp.~399--415.

\bibitem{laurent2007local}
{\sc T.~Laurent}, {\em Local and global existence for an aggregation equation},
  Communications in Partial Differential Equations, 32 (2007), pp.~1941--1964.

\bibitem{LLL18}
{\sc A.~Liu, J.-G. Liu, and Y.~Lu}, {\em On the convergence of empirical
  measures in $\infty$-transportation distance for unbounded densities}, arXiv
  preprint arXiv:1807.08365,  (2018).

\bibitem{Liu17}
{\sc Q.~Liu}, {\em Stein variational gradient descent as gradient flow},
  Advances in Neural Information Processing Systems (NIPS 2017), 30 (2017).

\bibitem{LiuWang16}
{\sc Q.~Liu and D.~Wang}, {\em Stein variational gradient descent: A general
  purpose bayesian inference algorithm}, Advances in Neural Information
  Processing Systems (NIPS 2016), 29 (2016).

\bibitem{markowich2000trend}
{\sc P.~A. Markowich and C.~Villani}, {\em {On the trend to equilibrium for the
  Fokker-Planck equation: an interplay between physics and functional
  analysis}}, Mat. Contemp, 19 (2000), pp.~1--29.

\bibitem{Otto01_porous}
{\sc F.~Otto}, {\em The geometry of dissipative evolution equations: The porous
  medium equation}, Communications in Partial Differential Equations, 26
  (2001), pp.~101--174.

\bibitem{raviart1985analysis}
{\sc P.-A. Raviart}, {\em An analysis of particle methods}, in Numerical
  methods in fluid dynamics, Springer, 1985, pp.~243--324.

\bibitem{roberts1996exponential}
{\sc G.~O. Roberts and R.~Tweedie}, {\em Exponential convergence of {L}angevin
  distributions and their discrete approximations}, Bernoulli, 2 (1996),
  pp.~341--363.

\bibitem{Teschl12}
{\sc G.~Teschl}, {\em Ordinary differential equations and dynamical systems},
  vol.~140 of Graduate Studies in Mathematics, American Mathematical Society,
  Providence, RI, 2012.

\bibitem{trillos2014rate}
{\sc N.~G. Trillos and D.~Slep\'cev}, {\em On the rate of convergence of
  empirical measures in $\infty$-transportation distance}, Canadian Journal of
  Mathematics, 67 (2014), pp.~1358--1383.

\bibitem{Vil03}
{\sc C.~Villani}, {\em Topics in Optimal Transportation}, vol.~58 of Graduate
  Studies in Mathematics, American Mathematical Society, Providence, RI, 2003.

\bibitem{weed2017sharp}
{\sc J.~Weed and F.~Bach}, {\em Sharp asymptotic and finite-sample rates of
  convergence of empirical measures in {W}asserstein distance}, arXiv preprint
  arXiv:1707.00087,  (2017).

\end{thebibliography}

\end{document}